\documentclass[12pt]{svjour3}
\usepackage{amssymb,amsmath}
\usepackage{graphicx}
\usepackage{color,cite}

\newcommand\be{\begin{equation}}
\newcommand\ee{\end{equation}}
\newcommand\bea{\begin{eqnarray}}
\newcommand\eea{\end{eqnarray}}
\newcommand\bi{\begin{itemize}}
\newcommand\ei{\end{itemize}}
\newcommand\ben{\begin{enumerate}}
\newcommand\een{\end{enumerate}}
\newcommand\bc{\begin{center}}
\newcommand\ec{\end{center}}
\newcommand\ba{\begin{array}}
\newcommand\ea{\end{array}}

\newcommand{\R}{\ensuremath{\mathbb{R}}}
\newcommand{\C}{\ensuremath{\mathbb{C}}}

\newcommand{\ga}{\alpha}                  
\newcommand{\gl}{\lambda}
\newcommand{\gd}{\delta}

\newcommand{\gO}{\Omega}

\newcommand{\dgO}{\partial\Omega}

\newcommand{\gep}{\epsilon}

\newcommand{\esig}{e^{\sigma}}
\newcommand{\esigz}{e^{\sigma_0}}

\newcommand{\egz}{e^{\gamma_0}}

\newcommand{\Lad}{L^\infty_\mathrm{ad}(\gO)}

\newcommand{\boldrho}{\boldsymbol\rho}

\newcommand{\boldk}{\mathbf{k}}
\newcommand{\kperp}{\mathbf{k}^\perp}

\newcommand{\uboldrho}{u_{\boldsymbol\rho}}
\newcommand{\boldv}{\mathbf{v}}
\newcommand{\boldw}{\mathbf{w}}
\newcommand{\boldeo}{\mathbf{e}_1}
\newcommand{\boldet}{\mathbf{e}_2}
\newcommand{\boldeth}{\mathbf{e}_3}
\newcommand{\boldej}{\mathbf{e}_j}
\newcommand{\boldem}{\mathbf{e}_m}

\newcommand{\usigf}{u_{\sigma,f}}
\newcommand{\Asigf}{A_{\sigma,f}}
\newcommand{\boldAsigf}{\mathbf{A}_{\sigma,\mathbf{f}}}
\newcommand{\boldAsiggamf}{\mathbf{A}_{\sigma,\gamma,\mathbf{f}}}

\newcommand{\CzOprime}{C_0(\overline{\Omega'})}
\newcommand{\CinfzOprime}{C^\infty_0(\overline{\Omega'})}
\newcommand{\Real}{\mathrm{Re\:}}

\providecommand{\norm}[1]{\lVert#1\rVert}

\begin{document}

\title{Stabilizing inverse problems by internal data. II. Non-local internal data and generic linearized uniqueness}
\titlerunning{Inverse problems with internal data. II}

\author{Peter Kuchment         \and
        Dustin Steinhauer 
}


\institute{Peter Kuchment \at
              Mathematics Department, Texas A\&M University, College Station, TX 77843-3368 \\
              \email{kuchment@math.tamu.edu}           
           \and
           Dustin Steinhauer \at
              Mathematics Department, Texas A\&M University, College Station, TX 77843-3368\\
              \email{dsteinha@math.tamu.edu}
}
\maketitle

\date
\begin{abstract}
In the previous paper \cite{KuchStein}, the authors introduced a simple procedure that allows one to detect whether and explain why internal information arising in several novel coupled physics (hybrid) imaging modalities could turn extremely unstable techniques, such as optical tomography or electrical impedance tomography, into stable, good-resolution procedures. It was shown that in all cases of interest, the Fr\'echet derivative of the forward mapping is a pseudo-differential operator with an explicitly computable principal symbol. If one can set up the imaging procedure in such a way that the symbol is elliptic, this would indicate that the problem was stabilized. In the cases when the symbol is not elliptic, the technique suggests how to change the procedure (e.g., by adding extra measurements) to achieve ellipticity.

In this article, we consider the situation arising in \emph{acousto-optical tomography} (also called \emph{ultrasound modulated optical tomography}), where the internal data available involves the Green's function, and thus depends globally on the unknown parameter(s) of the equation and its solution. It is shown that the technique of \cite{KuchStein} can be successfully adopted to this situation as well. A significant part of the article is devoted to results on generic uniqueness for the linearized problem in a variety of situations, including those
arising in acousto-electric and quantitative photoacoustic tomography.
\end{abstract}

\section*{Introduction}
In  \cite{KuchStein}, the authors introduced a simple technique that allows one to see whether a linearized hybrid imaging problem is elliptic, and if not, what additional information can make it so. It consists of the following steps: proving Fr\'echet differentiability, computing the derivative to discover that it is a pseudo-differential operator with an explicitly-determined principal symbol, and finally checking the ellipticity of that operator. This provides a simple, easy-to-apply, uniform view of all cases of interest that we have tried, which used to be considered separately and with different techniques. See, e.g., \cite{Ammari_AET,Ammari_book,bal1,bal3,cap,KuKuAET,widlak} for {\it acousto-electric tomography} (AET), and \cite{bal4,BalRenQPAT,BalUhl,BalExplicit,CoxQPAT} for {\it quantitative
photoacoustic tomography} (QPAT).

Since  \cite{KuchStein}, quite a few works have been dedicated to various extensions, applications, and further studies. One may see, e.g. \cite{BalInternal,Kuch_hybrid,KuRadon} for general overviews and references. As a non-exhaustive list of some recent works one can mention, for instance \cite{Alb_thesis, Alb1, Ammari_elast, Tam, Monard_thesis, Veras_thesis, BalSchBio,Montalto,WidScherz,MosSch,HoffKnu,AleGlobal,MonSt,Ren,bal3,BalGuoMon,BalGuoAnisCurrent,MonBalN3,MonBalD2,MonBalInvDiff,BalZhouMaxwell,
BalMosk,BalInternal,bal1,BalUhlmSolutions,BalHybridRedundant,BalNaetar,MonSt}.

However, several issues were not addressed in \cite{KuchStein} and since: first of all, in some cases the internal information comes from a global functional of the coefficients and solution of the equation (e.g., its Green's function); secondly, the uniqueness of the linearized problem was not considered in such a general setting, and probably does not hold in the whole generality of \cite{KuchStein}; finally, even if the
linearized injectivity were proven, it would not immediately imply stability of the nonlinear problem (although it would be a strong hunch), since the (semi-~)~Fredholm property of the derivative and differentiability were proved in non-matching spaces. (See \cite{StUhl} for a technique of proving nonlinear stability in spite of this discrepancy and \cite{MonSt} for its application to hybrid imaging problems. Another application will be given in the forthcoming paper  \cite{Stein3}. One deals here with a situation where advanced implicit function theorems could be useful \cite{Nash,Nirenberg72,Nirenberg81,HormImpl,HormImpl2}. 

The goal of this text is to overcome some of these deficiencies.

In \cite{KuchStein} we only considered the cases when
the internal information was provided as a function $F(\alpha(x),u(x),\nabla u(x))$,
computed at each internal point $x$, where $\alpha$ indicates here the parameter(s) of
the equation (i.e., conductivity, absorption coefficient, etc.) and $u$ its solution.
In so-called \emph{ultrasound modulated optical tomography} (UMOT)
\cite{AllmBang,bal4, BalMosk,BalSchAOT,Kuch_hybrid, wang2}, as well as in some versions of AET \cite{ScherAET}, internal values are provided for a function, which is dependent
on these variables in a nonlocal manner, e.g. through the Green's function of the equation.
We show in Theorem \ref{T:UMOT} of Section \ref{S:umot} that, at least in the UMOT situation, this does not prevent one
from employing the same simple linearization+microlocal analysis approach.

Then we switch to discussing the generic injectivity issue. In Section \ref{S:general}, we outline a technique based on theory of analytic Fredholm operator functions for proving generic injectivity by employing analytic dependence of the data on the parameters. It uses the fact that, under (semi-)Fredholmity conditions, non-injectivity can happen only at an analytic set of parameters. Thus, if one has a point where injectivity holds, then it has to hold almost everywhere in the connected component of this point (see Theorem \ref{T:ZKKP2} and following corollaries). This technique is then applied in Section \ref{S:uniq} to several examples arising in hybrid imaging methods, in particular in AET (Theorems \ref{T:testcase} and \ref{T:AET}) and QPAT (Theorem \ref{T:QPAT}). The main difficulty here is to figure out the connected component of the good parameters. What helps here is that one can go into the space of complex parameters (wherever Fredholmity is still preserved) and look into the connected component 
there. In other words, two sets of real parameters can be
connected through a complex path. In particular, complex geometrical optics (CGO) solutions are used to achieve this. Section \ref{S:lemmas} contains the proofs of some technical statements. Finally, Sections \ref{S:remarks} and \ref{S:acknow} are devoted to final remarks and acknowledgments, correspondingly.

\section{Ultrasound modulated optical tomography (UMOT): functionals
involving Green's function}\label{S:umot}

We start with introducing some notations used throughout this text. Let $\gO\subset\R^3$ (in some cases, when indicated, higher dimensional situations are also considered) be a smooth bounded region with relatively compact smooth subregions $\gO'\Subset\gO''\Subset\gO$. We consider the operator
\be \label{main_op}
L_\mu (x,D_x)u:=(-\Delta +e^{\mu(x)})u(x),
\ee
where we denote by $e^{\mu(x)}$ the \emph{absorption coefficient}. The \emph{log-absorption coefficient }$\mu(x)$ is used, since the considerations will be then much simpler, in comparison with when dealing with the absorption coefficient subjected to the positivity constraint.

Let also $B$ and $C$ be two boundary-value operators,
each of which is either Dirichlet, or Robin of the form
$D_\nu+\ga(x)$ for a nonnegative function $\ga\in C^\infty(\dgO)$
that is not identically zero.
Then the boundary value problems

\be \label{E:BVP_b}
\begin{cases}
L_\mu u=f \mbox { in }\Omega\\
Bu|_{\partial\Omega}=g
\end{cases}
\ee
and
\be \label{E:BVP_c}
\begin{cases}
L_\mu u=f \mbox { in }\Omega\\
Cu|_{\partial\Omega}=g
\end{cases}
\ee
are elliptic (see the general discussion of the so-called Shapiro-Lopatinsky, or covering conditions that guarantee ellipticity, for instance, in \cite{Lions}).

In UMOT with coherent light illumination, one needs to recover the (log-)absorption $\mu(x)$ using the data
\be \label{E:umotdata} 
F_\mu(\xi)=u_\mu(\xi)G_\mu(\eta,\xi)
\ee
that has been obtained from measurements \cite{AllmBang}.
Here $\xi\in\Omega$ is an arbitrary point in the interior of the domain, $\eta\in\partial\Omega$ is a fixed location of a boundary detector, and the time-averaged light intensity $u$ solves the boundary value problem
\be \label{E:boundarysource}
\begin{cases}
L_\mu u_\mu =0\\
Bu_\mu |_{\dgO}= S(x),
\end{cases}
\ee
with $S(x)\geq0$ describing the (given) light intensity of a source located at the boundary $\partial\Omega$. The function $G_\mu(x,\xi)$ is the Green's function of the boundary value problem (\ref{E:BVP_c}), and the subscript $\mu$ is used to indicate the dependence of the solution $u$ and Green's function on the coefficient of the equation\footnote{Since there has been some controversy about whether the boundary operators $B$ and $C$ coincide, we allow them to be different, which does not influence the results.}, i.e.:
\be \label{E:maineq}
\begin{cases}
L_\mu (x,D_x)G_\mu (x,\xi)=\gd(x-\xi),\quad x, \xi \in\Omega,\\
CG_\mu(y,\xi)=0, y\in\partial\Omega.
\end{cases}
\ee
We have 
\be \label{E:maineqxitoo}
L_\mu (\xi,D_\xi) G_\mu(x,\xi)=\gd(x-\xi)
\ee
as well \cite{LionsNotes}.

We will need the following lemma regarding the positivity of $G_\mu$.

\begin{lemma}\label{L:positive}
 Let $\eta\in\dgO$.  Then there exists a constant $c>0$ such that $G_\mu(\eta,\xi)\geq c$
 for $\xi\in\gO'$.
\end{lemma}

The proof is given in Section \ref{S:lemmas}.


Since $u$ represents the time-averaged light intensity within the object of interest,
our model should imply that $u$ is also a nonnegative function.  To enforce this,
if $B$ is a Robin boundary value operator we add an additional assumption that $u$
is nonnegative on $\dgO$.  Then the maximum
principle implies that $u$ is nonnegative in $\gO$ (in fact, the strong version
of the maximum principle implies that $u$ is strictly positive as long as $S$ is not
identically zero).
If $B$ is Dirichlet, we do not need to add any additional assumptions
to reach these conclusions.

We define the class of ``admissible'' functions $\mu(x)$ as follows:
\be
\Lad:=\{\mu\in L^\infty(\gO)\;\big|\;\mu|_{\gO\backslash\gO'}=0\}.
\ee
The reader notices that this class of functions forces the values of the absorption $\exp(\mu)$ near the boundary to be constant. In turn, this will allow us to work somewhat away from the boundary, which makes things simpler. One can generalize to the case of known (variable) values of $\mu$ near the boundary. However, the well-developed theory of overdetermined elliptic boundary value problems (see \cite{Solon,Spencer,Samb,GudKrein} and Section \ref{S:remarks}) should allow one to relax this condition even further (as it was done, for instance, in \cite{BalHybridRedundant}). We will also have to relax this requirement in Section \ref{S:uniquenessQPAT}.

\begin{lemma}\label{L:green}
Let $\eta\in\dgO$.  Then the map
\be
\mu\in\Lad \rightarrow G_\mu(\eta,\cdot)\in H^2(\gO'').
\ee
is Fr\'echet differentiable at any fixed $\mu_0\in  \Lad$.
\end{lemma}
The proof of this statement can be found in Section \ref{S:lemmas}.

Let now $\chi(x)$ be a smooth function, equal to $1$ in a neighborhood of $\overline{\Omega' }$ and zero outside $\Omega''$. We also denote by $dF$ the Fr\'echet derivative guaranteed by the previous lemma.

We can now formulate the main result of this section:
\begin{theorem}\label{T:UMOT}
If $F(\mu)(\xi)=u_\mu(\xi)G_\mu(\eta,\xi)$, then
\begin{enumerate}
 \item $\chi dF\chi$ is a pseudo-differential operator of order $-2$, elliptic in a neighborhood of $\overline{\gO'}$;

 \item $dF$ is Fredholm as an operator from $L^2(\gO')$ into $H^2(\gO')$;

 \item The kernel  $K\subset L^2(\gO')$ of $dF$ is finite dimensional.

 \item The operator $dF$ generates a topological isomorphism from $L^2(\gO')/K$ into $H^2(\gO')$, i.e.
there is a constant $C$ such that

\be \frac{1}{C} \norm{\mu_1}_{L^2(\gO')/K} \leq \norm{dF(\mu_1)}_{H^2(\gO')} \leq C\norm{\mu_1}_{L^2(\gO')/K} \ee
\end{enumerate}
\end{theorem}
\begin{proof}
Having established the Fr\'echet differentiability of the Green function with respect to $\mu$ in Lemma \ref{L:green}, we can find the derivative $dF$ by a formal calculation.  Consider a small perturbation of $\mu_0$ and the corresponding perturbation of the Green's function:

\be\label{E:gperturb}
\begin{cases}
\mu=\mu_0+\gep\mu_1\\
G_\mu(\eta,\xi)=G_0(\eta,\xi)+\gep G_1(\eta,\xi)+o(\gep),
\end{cases}
\ee
where $G_0(x,\xi)$ solves the boundary value problem (\ref{E:maineq}) with the coefficient $\mu_0$.
The elliptic regularity and smoothness of $\mu_0$ imply that $G_0(\eta,\xi)$ is a smooth function on $\gO'$.

Since the function $G_\mu$ satisfies the equation
\be
L_\mu G_\mu(\eta,\cdot)=0,
\ee
we can use (\ref{E:gperturb}) to find that $G_1$ solves the equation

\be
L_{\mu_0}(\xi,D_\xi) G_{\mu_1}(\eta,\xi)=-\mu_1 (x)e^{\mu_0(x)}G_0(\eta,\xi)
\ee
on $\gO'$.

Let $\zeta$ be the dual (Fourier) variable to $\xi$.
The operator $\chi L_{\mu_0}\chi$ has a parametrix with the principal symbol $\chi^2(\xi)(\zeta)^{-2}$.

Let $u_\mu(x)$ satisfy (\ref{E:boundarysource}). The mapping $\mu\mapsto u_\mu$ is Fr\'echet differentiable.  This fact is well known, and a proof can be found
as a special case of Lemma 2.1 in \cite{KuchStein} when the boundary condition in equation (\ref{E:BVP_b}) is Dirichlet.  The proof
given there generalizes easily to the other boundary conditions we allow.  The derivative $u_{(1)}$ comes from the formal expansion

\be u_\mu=u_0+\gep u_{(1)}+o(\gep) \ee
as in (\ref{E:gperturb}).  Here, $u_0$ solves (\ref{E:boundarysource}) with coefficient $\mu_0$.
Thus the mapping $\mu\mapsto F(\mu)$ is a Fr\'echet differentiable mapping from $\Lad\rightarrow L^2(\gO')$,
and $A(\xi,D_\xi)(\mu)=\chi dF(\chi\mu)$ as a pseudo-differential operator on $\R^n$ has principal symbol

\be A(\xi,\zeta)=-\frac{2\chi^2(\xi) e^{\mu_0}(\xi) u_0(\xi) G_0(\eta,\xi)}{\zeta^2}\;. \ee
(We remind the reader that we are using $\zeta$ as the dual variable to $\xi$.)
Both $u_0$ and $G_0$ are bounded below by positive constants on $\gO'$ by Lemma \ref{L:positive},
so $A(\xi,D)$ is elliptic on $\gO'$ of order $-2$.  The rest of the conclusions immediately follow.\qed
\end{proof}

We are now changing gears, switching to the problem of generic linearized uniqueness. In the next Section we list some known abstract notions and results that we will rely upon. One can find this information in various sources, e.g. in \cite{GoKr,Kato,Krein,ZKKP}.

\section{Analytic Operator Function Preliminaries}\label{S:general}

Theorem \ref{T:UMOT} and results of \cite{KuchStein} show that linearizations of various functionals arising in internal data problems
are Fredholm or left semi-Fredholm operators in appropriate Banach spaces (all Banach spaces here will be assumed being complex). We thus need to recall some definitions and facts from the theory of such operators and operator-valued functions.
\begin{definition}\indent
\begin{itemize}
\item A continuous linear operator $A\in L(E,F)$ between two Banach spaces is said to be \textbf{Fredholm}, if it has closed range and finite-dimensional kernel and co-kernel.
\item It is called \textbf{left semi-Fredholm} if it has finite-dimensional kernel and closed and complementable\footnote{I.e., having a closed complementary subspace. This condition is satisfied automatically in Hilbert spaces.} range.
\item We denote the spaces of Fredholm and left semi-Fredholm operators acting between Banach spaces $E$ and $F$ by $\Phi(E,F)$ and $\Phi_l(E,F)$, respectively.
\end{itemize}
\end{definition}
Another interpretation, useful when working with pseudo-differential operators, due to the parametrix construction, is in the following well-known proposition:
\begin{lemma}\label{L:fredholm}\indent
\begin{itemize}
\item An operator $A\in L(E,F)$ is Fredholm, iff it has a (two-sided) \textbf{regularizer}, i.e. an operator $B\in L(F,E)$ such that operators $AB-I$ and $BA-I$ are compact.
\item An operator $A\in L(E,F)$ is left semi-Fredholm iff it has a \textbf{left regularizer}, i.e. an operator $B\in L(F,E)$ such that operator $BA-I$ is compact\footnote{This claim applies to any Banach space, since in the definition of being left Fredholm operator we required complementability of the range. Without this requirement, existence of the left regularizer would not necessarily hold in spaces not isomorphic to Hilbert ones, although the converse statement would still be correct.}.
\end{itemize}
\end{lemma}

It is easy to derive from this lemma the following statement:
\begin{corollary}\label{C:vector}
If an operator $A_1\in L(E,F_1)$ is left semi-Fredholm, then the vector operator
$$
A:=(A_1, \ldots, A_k) \in L(E,\bigoplus\limits_{j=1}^k E_k)
$$
is also left semi-Fredholm.
\end{corollary}
We will use a more detailed version of this statement for pseudo-differential operators later.

As it happens, the operators arising in this text, as well as in \cite{KuchStein}, depend analytically on the coefficients of the equation under study and on the boundary data used. One deals here with infinite-dimensional analyticity, since these data belong to (complex) function spaces. Because we are interested in injectivity of these operators for generic coefficients and boundary values, the following fact, which is a special case of \cite[Theorem 4.13]{ZKKP}, comes in handy:

\begin{theorem}
 \label{T:ZKKP2}
Let $X$ be a connected Banach analytic manifold (e.g., a connected open domain in a complex Banach space) and let $E$ and $F$ be complex Banach spaces.
\ben
\item
Let $A:X\rightarrow\Phi(E, F)$ be an analytic
map, such that $A(z_0)$ is invertible for some $z_0\in X$.  Then $A(z)$ is invertible (and thus has zero kernel) except for $z$ lying in a proper analytic
subset of $X$.
\item
Let $A:X\rightarrow\Phi_l(E, F)$ be an analytic
map such that $A(z_0)$ is left-invertible for some $z_0\in X$.  Then $A(z)$ is left-invertible (and thus has zero kernel) except for $z$ lying in a proper analytic subset of $X$.
\een
\end{theorem}
 Here we used the following definition:
 \begin{definition} A set $Y\subset X$ is said to be \emph{analytic}, if it can be locally represented as the set of common zeros of a family of analytic functions. It is \emph{proper}, if at least one of these functions is not identically equal to zero (in other words, a proper analytic subset has a positive codimension in $X$).
 \end{definition}

The original, stronger version of theorem \ref{T:ZKKP2}, provided in \cite{ZKKP}, requires $X$ to be a Stein manifold, which would be an impediment in our case, due to the non-existence of Stein infinite dimensional Banach manifolds. For its local part stated above, however, $X$ does not have to be Stein.

%

We will use the results of Theorem \ref{T:ZKKP2} in the following clearly equivalent form:
\begin{corollary}\indent
\ben
\item
Let $A:X\rightarrow\Phi(E, F)$ be an analytic
map, such that $A(z_0)$ is injective for some $z_0\in X$.  Then $A(z)$ is \textbf{generically} injective, i.e. the set of points $z\in X$ where $A(z)$ is non-injective is a proper analytic subset of $X$.
\item
Let $A:X\rightarrow\Phi_l(E, F)$ be an analytic
map such that $A(z_0)$ is injective for some $z_0\in X$.  Then $A(z)$ is \textbf{generically} injective, i.e. the set of points $z\in X$ where $A(z)$ is non-injective is a proper analytic subset of $X$.
\een
\end{corollary}

Another corollary that we will need deals with the real situation (since eventually we need results dealing with real functional parameters).
Namely, we will be interested in the case when $Y$ is a connected open domain in a complex Banach space $E_C$ that is complexification $E_C=E+iE$ of a real Banach space $E$. We assume that $Y$ has a non-empty intersection with the real subspace.
For an open set $Y$ in such a Banach space, we denote by $\Real Y$ the intersection of $Y$ with $E$.
\begin{corollary}\label{C:real} Under the conditions above,
\begin{enumerate}
\item
Let $A:\Real Y\rightarrow\Phi(E, F)$ be an analytic
map, such that $A(z_0)$ is injective for some $z_0\in \Real Y$.  Then $A(z)$ is \textbf{generically} injective in $\Real Y$, i.e. the set of points $z\in \Real Y$ where $A(z)$ is non-injective is a proper (i.e., of positive codimension) analytic subset of $\Real Y$.
\item
Let $A:\Real Y\rightarrow\Phi_l(E, F)$ be an analytic
map such that $A(z_0)$ is injective for some $z_0\in \Real Y$.  Then $A(z)$ is \textbf{generically} injective in $\Real Y$, i.e. the set of points $z\in \Real Y$ where $A(z)$ is non-injective is a proper (i.e., of positive codimension) analytic subset of $\Real Y$.
\end{enumerate}
\end{corollary}
This statement follows from the observations that, first, the intersection of the non-injectivity set with the real subspace $E$ is real analytic and, second, if this real analytic set contains an open subset in $E$, then a simple analytic continuation argument shows that the non-injectivity set must cover the whole $X$, which is a contradiction. \qed

\begin{remark}
An analog of this corollary also holds, without any essential change in the proof, if $E_C$ is replaced by a connected complex analytic Banach manifold and $E$ with a maximal totally real real-analytic submanifold in $E_C$.
\end{remark}

It will be sometimes easier for us to establish analyticity of the linearizations with values in a larger space of operators (i.e., in a weaker norm) than what we will need.  This will not cause any problems, due to the following simple lemma:

\begin{lemma} \label{L:ana_wrong_spaces}
Let $X$ be a complex analytic Banach manifold.  Let $E_1$, $E_2$, $F_1$, and $F_2$ be complex Banach spaces such that there are dense continuous embeddings
$E_1\hookrightarrow E_2$ and $F_2\hookrightarrow F_1$.  Let $A:X\rightarrow L(E_1,F_1)$
be an analytic map that is also locally uniformly bounded (in the operator norm) as a map from $X$ to $L(E_2,F_2)$.  Then
$A$ is an analytic map from $X$ to $L(E_2,F_2)$.
\end{lemma}
The assumption that $A$ is locally uniformly bounded means that for every $z\in X$ there exist $M(z)$, $r(z)>0$
such that $\norm{A(w)}_{L(E_2,F_2)}<M(z)$ if $\norm{w-z}<r(z)$.

\begin{proof}
It suffices to show that $A$ is weakly analytic into $L(E_2,F_2)$, namely for every $e\in E_2$ and
$f^*\in F_2^*$ the function $<f^*,A(z)e>$ is analytic \cite[Chapter 3, Theorem 1.37]{Kato}.

Given $e\in E_2$ and $f^*\in F_2^*$,
let ${e_n}\in E_1$ be a sequence converging to $e$ and let $f_n^*\in F_1^*$ converging to $f^*$.
Since $A:X\rightarrow L(E_1,F_1)$ is analytic, the functions $<f_n^*,A(z)e_n>$ are analytic in $X$.
Furthermore, we have

\bea |<f_n^*,A(z)e_n>&-&<f^*,A(z)e>|\nonumber\\
&=&|<f_n^*-f^*,A(z)e_n>+<f^*,A(z)(e_n-e)>|\nonumber\\
&\leq&\norm{f_n^*-f^*}_{F_2^*}\norm{A(z)}_{L(E_2,F_2)}\norm{e_n}_{E_2}\nonumber\\
&+&\norm{f^*}_{F_2^*}\norm{A(z)}_{L(E_2,F_2)}\norm{e_n-e}_{E_2}   \; .\eea

Since $\norm{A(z)}_{L(E_2,F_2)}$ is locally bounded, the functions  $<f_n^*,A(z)e_n>$ converge
locally uniformly to $<f^*,A(z)e>$.
Hence the limit $<f^*,A(z)e>$ is analytic. \end{proof}

We will now use these abstract results for proving generic linearized uniqueness for some hybrid imaging problems.

\section{Generic linearized uniqueness in hybrid imaging problems}\label{S:uniq}

The injectivity for the linearized operators most probably does not hold in the wide generality of \cite{KuchStein} or Section \ref{S:umot}. However, it is expected to hold for ``generic'' parameters of the problems under consideration. Proving this is the goal of the section.

The word ``genericity'' can mean different things. It could be used in terms of Baire category, or ``almost everywhere'' (in a space with a measure), or in the meaning of ``with probability one'' (in a space with a probability measure). Another, stronger, level is reached in ``open and dense subset'' genericity. Probably the strongest and most productive is ``except for an analytic set,'' since it allows one to use transversality theorems (see, e.g., \cite{Lang,arnold}) to make conclusions for generic families--not just single operators. We aim for this stronger version, but our technique of pseudo-differential operators with infinitely smooth symbols happens to be an obstacle here. As it will be shown in the next publication \cite{Stein3}, this obstacle comes from the techniques used, rather from the substance. Using pseudo-differential calculi with symbols of finite smoothness (see, e.g., \cite{TaylorTools}) resolves this. In this paper, though, we will not go that far and stop at the ``open and dense''
set
level.

\subsection{Hybrid inverse conductivity problems}

In inverse conductivity problems one is concerned with recovering the \emph{log-conductivity} $\sigma$ in the boundary value problem
\be\label{E:invcondeqn}
\begin{cases}
-\nabla\cdot(e^\sigma\nabla\usigf)=0\\
\usigf|_{\dgO}=f\;.
\end{cases}
\ee
(We use the notation $\usigf$ to emphasize the dependence of the solution $u$ on $(\sigma, f)$.)
We are interested in recovering $\sigma$ from some internal data. In many inverse conductivity problems one ends up having the interior data  of the form.
\be\label{E:newF}
F(\sigma)(x)=e^{2\sigma(x)/p}|\nabla \usigf(x)|^2
\ee
for $p>0$ fixed.
\begin{remark}
In
\cite{KuchStein} and in the majority of the literature, see e.g. \cite{Kuch_hybrid,BalExplicit},
the data functionals
\be\label{E:oldF}
\tilde{F}(\sigma)(x)=e^{\sigma(x)} |\nabla \usigf(x)|^p
\ee
were studied. It was noticed that the different ranges of values of $p$ lead to rather different techniques and indeed results. However, raising the expression (\ref{E:oldF}) to the power $2/p$, one arrives at (\ref{E:newF}). Although this does not eliminate dependence on $p$ in various results, it shows that some technical difficulties in dealing with  (\ref{E:oldF}) were artifacts of the form the expression was written. This is due to the more benign (indeed, quadratic) dependence of (\ref{E:newF}) on $\nabla \usigf$.
\end{remark}

Given a smooth log-conductivity $\sigma\in\Lad$ and $f\in H^{1/2}(\dgO)$, the corresponding linearization $\Asigf$ takes the form

\be \label{E:Asigma}
\Asigf(\rho)=dF(\rho)=\frac{2}{p} e^{2\sigma/p}\left(\rho|\nabla \usigf|^2+p\nabla \usigf\cdot\nabla v(\rho)\right),
\ee
where $\rho\in L^2(\gO')$ and $v(\rho)\in H^1_0(\gO)$ solves the equation
\be \label{E:v}
-\nabla\cdot(e^\sigma\nabla v(\rho))=\nabla\cdot(\rho e^\sigma\nabla \usigf).
\ee
We can now allow the log-conductivity $\sigma(x)$ and boundary value $f$ to be complex (which does not undermine ellipticity of the problem). Using the notation $a\cdot b=\sum a_jb_j$ for the bilinear product of complex vectors, we can rewrite (\ref{E:Asigma}) for complex values of parameters $\sigma$ and $f$ as
\be \label{E:Asigma_comple}
\Asigf(\rho)=dF(\rho)=\frac{2}{p} e^{2\sigma/p}\left(\rho\nabla \usigf\cdot\nabla \usigf +p\nabla \usigf\cdot\nabla v(\rho)\right).
\ee
Then
one can establish the analytic dependence of $\Asigf$ on $(\sigma, f)$.

\begin{lemma}\label{L:dependence}
 The map

\bea \Lad\times H^{1/2}(\dgO)&\rightarrow& L(\Lad,L^1(\gO')) \nonumber\\
(\sigma,f)&\mapsto& \Asigf \eea
is analytic.
\end{lemma}

This lemma is proved in Section \ref{S:lemmas}.

Now, as in \cite{KuchStein} and other studies, the considerations and results start depending on the value of $p$. We thus concentrate on various ranges of values of $p$.

\subsection{The case when $0<p<1$}
\indent

This is, as it was seen in \cite{KuchStein,MonSt}, the simplest (although, maybe the least applicable) situation.

Our approach to investigating the invertibility of $dF$ will depend on the dimension.
Indeed, the situation is simpler in dimension 2 than in higher dimensions. The reason is that in dimension $2$ it is possible to
select two boundary conditions $f_j$ in (\ref{E:invcondeqn}) such that for any $\sigma$ the gradients of the corresponding solutions
are linearly independent \cite{ale}.  Such a choice is not always possible in higher dimensions \cite{lau}.

When $n=2$,
we will assign a boundary condition $f=x_1$.
After showing that $A_{0,x_1}$ is invertible,
an application of Theorem \ref{T:ZKKP2} will
then show that $\Asigf$ is invertible on an open and dense subset of $\sigma\in C_0(\gO')$.

When $n\geq 3$, we might need more measurements..
Namely, let $m\geq n$ and let us denote by $\mathbf{f}$ a set of $m$ Dirichlet boundary data $(f_1,\ldots, f_m)$. We introduce now a vector operator as follows:
\be
\boldAsigf = \left(\begin{array}{c}
A_{\sigma,f_1}\\
\vdots	\\
A_{\sigma,f_m}\\
\end{array}\right).
\ee
We define the following sets:
\begin{definition}\indent
\begin{itemize}
\item $X_m$ is the set of all real-valued pairs
$$
(\sigma,\mathbf{f})\in \Real\left(\CinfzOprime\times H^{1/2}(\dgO)^m\right)
$$ such that
the gradients
$$
\nabla u_{\sigma,f_1},\ldots, \nabla u_{\sigma,f_m}
$$
of the corresponding solutions of (\ref{E:invcondeqn}) span the whole space $\R^n$ at every point $x$ in
a neighborhood of $\overline{\gO''}$.

This will be the set of ``good'' $m$-tuples of measurements, for which left semi-Fredholmity holds.

\item $\overline{X_m}$ is the closure of $X_m$ in $\Real\left(\CzOprime\times H^{1/2}(\dgO)^m\right)$.
\item $Y_m$ is the set of (possibly complex-valued) pairs
$$
(\sigma,\mathbf{f})\in \CinfzOprime\times H^{1/2}(\dgO)^m
$$
 such that
 $$
 \boldAsigf\in\Phi_l(L^2(\gO),L^2(\gO)^m)\;.
 $$
 \item $Y_m^0$ is the connected component of $Y_m$ containing the point $(0,\mathbf{f}_0)$, where
 $\mathbf{f}_0=(x_1,\ldots,x_n,\ldots)$ with
 the second ``$\ldots$'' representing $m-n$ arbitrarily chosen real functions.
 \item $\overline{Y_m^0}$ is the closure of $Y_m^0$ in $C_0(\overline{\gO'})\times H^{1/2}(\dgO)^m$.
\end{itemize}
\end{definition}

The idea here is to show that any two points in the set $X_m$ can be connected by a path through the domain of complex material parameters, while preserving semi-Fredholmity. This is exactly what the next theorem claims.
\begin{theorem}\label{T:XY}
Let the sets $X_m$, $Y_m$, and $Y_m^0$ be as above.  Then
\begin{enumerate}
\item all of these sets are non-empty;
\item $X_m\subset Y_m^0$.
\end{enumerate}
\end{theorem}

Rather than proving Theorem \ref{T:XY}, we postpone its proof and derive from it the main result of this subsection:

\begin{theorem}\label{T:testcase}
\indent
\ben
\item Let $n=2$ and $f=x_1$.  Then the operator $\Asigf$ is invertible as an operator on $L^2(\gO')$ for an open dense set
of $\sigma\in \Real\CzOprime$.
\item Let $m\geq n\geq3$.
Then the operator
\be
\boldAsigf = \left(\begin{array}{c}
A_{\sigma,f_1}\\
\vdots	\\
A_{\sigma,f_m}\\
\end{array}\right)
\ee
is injective as an operator on $L^2(\gO')$ for an open dense set
of $(\sigma,f)\in \overline{X_m}$.
\een

\end{theorem}

\begin{proof}
Let us consider first the two-dimensional case.
We claim that for $\sigma=0$, $f=x_1$, operator $\Asigf$ is invertible.
Indeed, with this choice of boundary condition the operator then reduces to

\be \label{E:A1steaksauce} A_{0,x_1}(\rho)=\rho-p\partial_1\Delta^{-1}(\partial_1\rho)\;. \ee
(Here $\Delta^{-1}$ refers to the inverse of the Laplacian on $\gO$ with homogeneous Dirichlet boundary condition.)
The boundary value problem

\be
\begin{cases}
   \Delta\rho-p\partial_1^2\rho=\Delta(A_{0,x_1}(\rho)) \\
\rho|_{\dgO}=0
  \end{cases}
\ee
has a unique solution in $L^2(\gO)$ for $A_{0,x_1}(\rho)$ given (e.g., \cite{BerKrRo}), establishing the invertibility of $A_{0,x_1}$.

Consider the operators $A_{\sigma,x_1}$. According to Lemma \ref{L:dependence}, they depend analytically
on $\sigma\in \CzOprime$
as a family of operators mapping $\Lad$ into $L^1(\gO')$.  By Lemma \ref{L:ana_wrong_spaces},
$A_{\sigma,x_1}$ is an analytic family of operators mapping $L^2(\gO')$ into itself. By \cite[Theorem 3.2]{KuchStein}, $A_{\sigma,x_1}\in\Phi(L^2(\gO'),L^2(\gO'))$ when $\sigma\in C^\infty_0(\gO')$.
Because the set of Fredholm operators is open in the operator norm topology,
there is an open dense set $V\subset \CzOprime$, containing all $\sigma\in\Real\CinfzOprime$, where the operators are also Fredholm.
Then the first statement of Theorem \ref{T:ZKKP2} applied to $\Real V$
(in the version of the first statement of Corollary \ref{C:real}) implies that there exists a set $W$, open and dense in $\mathrm{Re\:}\CzOprime$, where the operators are invertible.
This proves the first statement\footnote{In fact, we have proven somewhat more. Indeed, the (closed nowhere dense) complement of $W$ in $\mathrm{Re\:}V$ is an analytic set. Regrettably, the (closed and nowhere dense) complement of $\mathrm{Re\:}V$ in $\mathrm{Re\:}\CzOprime$ is not controllable. The second author will improve on this, getting the whole non-injectivity set analytic, later on in \cite{Stein3}, by using a non-smooth calculus of pseudo-differential operators.} of Theorem \ref{T:testcase}.

We proceed to proving the second statement of the theorem.
As in the previous part, according to Lemma \ref{L:dependence}, the operators $\boldAsigf$ depend analytically
on $(\sigma,\mathbf{f})\in \CzOprime\times H^{1/2}(\dgO)^m$
as a family of operators mapping $\Lad$ into $L^1(\gO')^m$.  By Lemma \ref{L:ana_wrong_spaces},
$\boldAsigf$ is an analytic family of operators on $Y_m^0$ mapping $L^2(\gO')$ into $L^2(\gO')^m$.
There exists a subset $V\subset \CzOprime\times H^{1/2}(\dgO)$, open and dense in $\overline{Y_m^0}$, such that $\boldAsigf\in \Phi_l(L^2(\gO'),L^2(\gO')^m)$ for $(\sigma,\mathbf{f})\in V$.
Since $X_m\subset Y_m^0$, we may assume that $V$ contains an open neighborhood of $X_m$ in $\Real(\CzOprime\times H^{1/2}(\dgO)^m)$.
As in the proof of the previous part we have that $A_{0,x_1}$ is an invertible operator, so $\Real Y_m^0$ contains a point $(\sigma,\mathbf{f})$ at which $\boldAsigf$ is injective.
Then the second statement of Theorem \ref{T:ZKKP2}
(in the version of the first statement of Corollary \ref{C:real}) applied to $\Real V$ implies that there exists a set $W$, open and dense in $\Real V$,
where the operators are injective. Since the restriction of an open dense set to a dense topological subspace is also open dense, $W\cap\overline{X_m}$ is open dense in $\overline{X_m}$,
proving the second statement of the theorem. \end{proof} \hfill $\Box$

We return now to proving Theorem \ref{T:XY}. Here, as well as in later sections,
we will make use of \textbf{complex geometrical optics} (\textbf{CGO}) solutions, as in \cite{BalUhl}.

Let $\boldrho\in\C^n$ satisfy
$
\boldrho\cdot\boldrho=0.
$
(As before, the dot product here and throughout
denotes the bilinear inner product $\boldv\cdot\boldw=v_1w_1+\ldots v_nw_n$ for $\boldv,\boldw\in\C^n$.)

The following statement is a direct consequence of the result of \cite[Prop. 3.3]{BalUhl}, as explained in \cite[Section 5.3]{BalInternal}:
\begin{proposition}\label{P:cgo}
There exists a CGO solution $u_{\boldrho}$ of (\ref{E:invcondeqn}), such that:
\be
u_{\boldrho} = e^{\boldrho\cdot x-\sigma(x)/2 }(1+\psi_{\boldrho}(x)),
\ee
where the remainder $\psi_{\boldrho}$ satisfies the equation
\be
\Delta \psi_{\boldrho}+2\boldrho\cdot\nabla\psi_{\boldrho}
=e^{-\sigma(x)/2 }\Delta \left(e^{\sigma(x)/2 }\right)(1+\psi_{\boldrho})
\ee
and the estimate
\be
\sup_{\overline{\gO}}|\rho \psi_{\boldrho}|\leq C\;.
\ee
The gradient of $u_{\boldrho}$ satisfies
\be\label{E:CGO_gradient}
\nabla u_{\boldrho}=e^{\boldrho\cdot x-\sigma(x)/2}(\boldrho+\boldsymbol{\phi}_{\boldrho}(x)),
\ee
and
\be
\sup_{\overline{\gO}} | \boldsymbol{\phi}_{\boldrho}|\leq C,
\ee
where $C$ is independent of $\boldrho$, for $\sigma$ in any bounded set in $H^{n/2+1+\gep}(\gO)$.
\end{proposition}

We will write $\boldrho=\rho(\boldk+i\bold{k}^\perp)/\sqrt{2}$ for real orthogonal unit vectors $\bold{k}$ and $\bold{k}^\perp$.
Letting $\theta(x)=\rho\bold{k}^\perp\cdot x/\sqrt2$, we then have

\be\label{E:CGO_im_gradient}
\mathrm{Im}(e^{\boldrho\cdot x}\boldrho) = \frac{\rho}{\sqrt2} e^{\frac{\rho}{\sqrt2}\bold{k}\cdot x}\left(\cos\theta(x)\kperp+\sin\theta(x)\boldk\right). \ee

We will also denote by $\uboldrho^I$ the imaginary part of $\uboldrho$
and by $f_{\sigma,\boldrho}^I$ the imaginary part of the restriction of $\uboldrho^I$
to $\dgO$.
For future reference, we note that if $u$ is a function in $C^1(\overline{\gO''})$ satisfying $u(x)\geq c>0$ on $\overline{\gO''}$, then since

\bea |u_{\boldrho}^I(x)\nabla u(x)| &\leq& \sqrt{2}e^{\frac{\rho}{\sqrt2}\boldk\cdot x-\sigma(x)/2}(1+\norm{\psi_{\boldrho}}_{L^\infty(\gO)})\norm{u}_{C^1(\overline{\gO''})}\;, \nonumber\\
|u(x)\nabla u_{\boldrho}^I(x)| &=& e^{\frac{\rho}{\sqrt2}\boldk\cdot x-\sigma(x)/2}|\boldrho+\boldsymbol{\phi}_{\boldrho}(x)||u(x)| \nonumber\\
&\geq& e^{\frac{\rho}{\sqrt2}\boldk\cdot x-\sigma(x)/2}\frac{\rho}{2}c\;, \nonumber\eea
we have

\be\label{E:CGOgradbigger}
u(x)\nabla u_{\boldrho}^I(x) - u_{\boldrho}^I(x)\nabla u(x)= u(x)\nabla u_{\boldrho}^I(x)\left(1+O\left(\frac{1}{\rho}\right)\right) \ee
as $\rho\to\infty$, the implied constant being uniform in $x\in\overline{\gO''}$.

{\it Proof of Theorem \ref{T:XY}.}

To prove the first statement, it is sufficient to prove the second statement and to show the nonemptiness of $X_m$. The latter is done by providing the example of $(0,\mathbf{f}_0)$ in the definition above.

Let us now prove that the whole set $X_m$ sits in the same connected component $Y^0_m$ of $Y_m$.

Let $(\sigma,f_1,\ldots,f_m)\in X_m$.
We use the following chain of deformations

\bea\label{E:deformation1}
(\sigma,f_1,\ldots,f_m) &\leadsto& (\sigma,if^I_{\sigma,\boldrho},\ldots,if^I_{\sigma,\boldrho}) \\
\label{E:deformation2} &\leadsto& (0,if^I_{0,\boldrho},\ldots,if^I_{0,\boldrho})\\
\label{E:deformation3} &\leadsto& (0,\mathbf{f}_0) \;. \eea
The first deformation (\ref{E:deformation1})
will be given by
\be\label{E:fjt}
(\sigma, f_{1,t} ,\ldots,f_{n,t})\quad 0\leq t\leq 1\;, \ee
where $f_{j,t} = (1-t)f_j + itf^I_{\sigma,\boldrho}$,
the second deformation (\ref{E:deformation2}) will be defined by letting $\sigma_t=(1-t)\sigma$, $0\leq t\leq1$,
and the third (\ref{E:deformation3}) is defined in a way similar to (\ref{E:deformation1}).

For sufficiently large $|\rho|$, the operator $\boldAsigf$ is left semi-Fredholm along the first and third deformations, which is a consequence of the following lemma.

\begin{lemma}
Let $(\sigma,\mathbf{f})\in X_m$, and let $f_{j,t}$ be as in (\ref{E:fjt}).  Then
$(\sigma,\mathbf{f}_{j,t})\in Y_m$ for $\rho=|\boldrho|$ sufficiently large.

\end{lemma}

{\it Proof of the lemma}.
By the assumption that $(\sigma,\mathbf{f})\in X_m$, the lemma is true for $t=0$.
We next examine the situation when $0<t<1$.
As shown in \cite{KuchStein}, a single operator $A_{\sigma,f}$ is a pseudo-differential operator with the principal symbol on $\gO''$

\be A_{\sigma,f}(x,\xi)=e^{2\sigma(x)/p}\frac{2}{p}\left(|\nabla u_{\sigma,f}(x)|^2-p\frac{(\nabla u_{\sigma,f}(x)\cdot\xi)^2}{|\xi|^2}\right)\;. \ee
When $\sigma$ and $f$ are complex-valued, this should be understood as

\be A_{\sigma,f}(x,\xi)=e^{2\sigma(x)/p}\frac{2}{p}\left(\nabla u_{\sigma,f}(x) \cdot\nabla u_{\sigma,f} -p\frac{(\nabla \usigf(x)\cdot\xi)^2}{|\xi|^2}\right)\;. \ee

We observe that $A_{\sigma,f}(x,\xi)$ is nonvanishing when, for all $\xi\in S^{n-1}$, we have

\be \label{E:cdnnot0}
\nabla \usigf \cdot\nabla\usigf -p(\nabla\usigf\cdot\xi)^2 \neq 0\;.\ee
We will show that for each $x\in\overline{\gO''}$, $\xi\in S^{n-1}$, and $0<t<1$, the inequality (\ref{E:cdnnot0}) is satisfied by $u_{j,t}
=u_{\sigma,f_{j,t}}$ for some $j$.
We will do this by showing that the left-hand side of (\ref{E:cdnnot0}) has nonvanishing imaginary part.
Using the simple identities

\be\label{E:id1} \mathrm{Im}[(\boldv +i\boldw)\cdot(\boldv +i\boldw)]=2\boldv\cdot\boldw \ee
and
\be
\label{E:id2} \mathrm{Im}[((\boldv +i\boldw)\cdot \xi)^2] = 2(\boldv\cdot\xi)(\boldw\cdot\xi)\ee
for real vectors $\boldv$, $\boldw$, and $\xi$,
we calculate that

\bea
\mathrm{Im} [\nabla u_{j,t}\cdot \nabla u_{j,t}]& =& (1-t)t\rho\sqrt{2}e^{\rho\boldk\cdot x/\sqrt{2}-\sigma(x)/2} \nonumber\\
\label{E:im1stpart}
&\times&(\nabla u_{\sigma,f_j}\cdot\kperp\cos \theta
+\nabla u_{\sigma,f_j} \cdot\boldk\sin \theta ) (1+o(1)) \\
\mathrm{Im} [(\nabla u_{j,t}\cdot\xi)^2]&=&(1-t)t\rho\sqrt{2}e^{\rho\boldk\cdot x/\sqrt{2}-\sigma(x)/2}\;,\nonumber\\
\label{E:im2ndpart}
&\times& (\nabla u_{\sigma,f_j}\cdot\xi)(\kperp\cdot\xi\cos\theta+\boldk\cdot\xi\sin\theta)(1+o(1))
\eea
as $\rho\to\infty$.
Combining this with (\ref{E:cdnnot0}), we find that the claim will be proved if we can show the inequality

\bea 0&\neq& \mathrm{Im} [\nabla u_{j,t}\cdot\nabla u_{j,t}-p(\nabla u_{j,t}\cdot\xi)^2] \nonumber\\
&=&(1-t)t\rho\sqrt{2}e^{\rho\boldk\cdot x/\sqrt{2}-\sigma(x)/2} \nonumber\\
\label{E:thebigzero}
&\times&\nabla u_{\sigma,f_j}\cdot\left[(\kperp-p(\xi\cdot\kperp)\xi)\cos\theta +(\boldk-p(\xi\cdot\boldk)\xi)\sin\theta\right]
\eea
and $\rho$ is taken sufficiently large.

Since we are assuming that $p<1$, the term in brackets is a nonzero vector.  To see this, assume without
loss of generality that $\boldk=\boldet$, $\kperp=\boldeo$, and let

\be
\omega_\theta := \left(\begin{array}{c}
                       \cos\theta\\
		       \sin\theta
                      \end{array}\right)\;.
\ee
Then the projection of the term in square brackets in (\ref{E:thebigzero}) onto the $\boldeo\boldet$-plane is

\be\label{E:proj12}
\omega_\theta-p\left(\left(\begin{array}{c}
                       \xi_1\\
		       \xi_2
                      \end{array}\right)\cdot\omega_\theta\right)\left(\begin{array}{c}
                       \xi_1\\
		       \xi_2
                      \end{array}\right)\;.\ee
Since
$\xi\in S^{n-1}$, $|(\xi_1,\xi_2)^t|\leq1$.
As $p<1$, we have that $|p((\xi_1,\xi_2)^t\cdot\omega_\theta)|<1$, and the claim is established.

Now,
the operator
$$\mathbf{A}_{\sigma,\mathbf{f}_t}
=\left(
\begin{array}{c}
A_{\sigma,f_{1,t}}\\
\ldots\\
A_{\sigma,f_{m,t}}
\end{array}
\right)
$$
has a left regularizer and thus is left semi-Fredholm if for all $x\in\overline{\gO''}$ and $\xi\in S^{n-1}$
there exists $j$ such that in a neighborhood of $(x,\xi)$ the principal symbol of $A_{\sigma,f_{j,t}}$ does not vanish
(see the construction of the left regularizer in \cite[proof of Theorem 4.1]{KuchStein}).
Since $(\sigma,\mathbf{f})\in X_m$, the vectors $\nabla u_{\sigma,f_j}(x)$ span $\R^n$ for any $x$ in a neighborhood
of $\overline{\gO''}$.
This means that for every $x$ in this neighborhood, the inequality (\ref{E:thebigzero}) is satisfied for at least one $j$;
hence $\mathbf{A}_{\sigma,\mathbf{f}_t}$ is a left semi-Fredholm operator. \emph{This finishes the proof of the lemma}.

To prove the theorem, it remains to establish that $(\sigma,\bold{f})\in Y_m$ along (\ref{E:deformation2}).  Along this deformation,
we observe from (\ref{E:CGO_gradient}) and (\ref{E:CGO_im_gradient}) that $\nabla u_{\boldrho}^I(x)\neq0$, because $\cos\theta(x)$ and $\sin\theta(x)$
cannot both vanish for a given $x$.  Therefore, $\mathbf{A}_{\sigma_t,f^I_{\sigma_t,\boldrho}}$ is left semi-Fredholm and thus $ (\sigma_t,\boldrho)\in Y_m$.

From this we conclude that $(\sigma,\mathbf{f})\in Y_m^0$. \hfill $\Box$

\subsection{Acousto-Electric Tomography. $p=2$}

We now turn to acousto-electric tomography (AET).
In a linearized version of the AET problem, the goal is to invert the functional of equation (\ref{E:Asigma})
for $p=2$:

\be \label{E:Asigma2}
\Asigf(\rho)=\esig\left(\rho|\nabla u_\sigma|^2+2\nabla u_\sigma\cdot\nabla v(\rho)\right),
\;\rho\in L^2(\gO').
\ee
We will assume access to
the three functionals

\be
\begin{cases} \label{E:2data}
A_{\sigma,f_1}(\rho)=\rho\esig|\nabla u_{\sigma,f_1}|^2+2\esig\nabla u_{\sigma,f_1}\cdot\nabla v^{(1)}(\rho)\\
A_{\sigma,f_2}(\rho)=\rho\esig|\nabla u_{\sigma,f_2}|^2+2\esig\nabla u_{\sigma,f_2}\cdot\nabla v^{(2)}(\rho)\\
A_{\sigma,(f_1,f_2)}(\rho)=\rho e^{2\sigma}|\nabla u_{\sigma,f_1}\cdot\nabla u_{\sigma,f_2}|^2\\
\indent +e^{2\sigma}\left(\nabla u_{\sigma,f_1}\cdot\nabla v^{(2)}(\rho) + \nabla u_{\sigma,f_2}\cdot\nabla v^{(1)}(\rho) \right)
\end{cases}
\ee
when $n=2$, and to more functionals (to be specified in a moment) when $n=3$.
In (\ref{E:2data}), $v^{(i)}$ solves equation (\ref{E:v}) with $u_{\sigma,f_i}$ in place of $u_{\sigma,f}$, $i=1,2$.
Such functionals as in (\ref{E:2data}) have been extracted from the measured data in hybrid imaging methods
(see
for example \cite{bal3,cap,KuKuAET,widlak}).

In \cite{KuchStein}, the map

\be
\left(
\ba{c}
A_{\sigma,f_1}\\
A_{\sigma,f_2}\\
A_{\sigma,(f_1,f_2)}
\ea
\right) :L^2(\gO')\rightarrow L^2(\gO')^3
\ee
was shown to be left semi-Fredholm.
In \cite{KuKuAET}, a left inverse was constructed for this operator when $n=2$ or 3 for $\sigma=0$ and the Dirichlet boundary data
$\mathbf{f}:=(f_1(x), f_2(x))=(x_1,x_2)$.  Though the proof of \cite{KuKuAET} extends to higher dimensions we will consider only the cases $n=2$ or 3 here.

Similarly to Theorem \ref{T:testcase}, when $n=3$ we will need to assume that we have more data than what was needed to establish
left semi-Fredholmity of the AET problem in \cite{KuchStein}\footnote{It is not quite clear at this moment how necessary it is to assume that many measurements.}.
Let $\mathbf{f}=\{f_j$, $j=1,\ldots,m\}$ be $m$ Dirichlet boundary data functions in (\ref{E:invcondeqn}),
and let

\be \boldAsigf=
\left(\begin{array}{c}
A_{\sigma,f_1}\\
\vdots\\
A_{\sigma,f_{m-1}}\\
A_{\sigma,(f_1,f_m)}
\end{array}\right)\;.
\ee
Analogously to the previous sub-section, we define the following sets for $m\geq4$:
\begin{definition}\indent
\begin{itemize}
\item $X_m$ is the set of $(\sigma,\mathbf{f})\in \Real (\CinfzOprime)\times H^{1/2}(\dgO)^m)$
such that at every $x\in\overline{\gO''}$
the sets of vectors $(\nabla u_1(x),\ldots, \nabla u_{m-1}(x))$ span $\R^n$ and $ (\nabla u_1(x),\nabla u_m(x))$
are linearly independent.
\item $\overline{X_m}$ is the closure of $X_m$ in $\Real\left(\CzOprime\times H^{1/2}(\dgO)^m\right)$.
\item $Y_m$ is the set of $(\sigma,\mathbf{f})\in \CinfzOprime\times H^{1/2}(\dgO)^m $
such that $\boldAsigf\in\Phi_l(L^2(\gO),L^2(\gO)^m)$.
\item $Y_m^0$ is the connected component of $ Y$ containing $(0,\mathbf{f}_0)$, where
$$\mathbf{f}_0=(x_1,x_2,x_3,\ldots,x_2)$$
(the dots ``$\ldots$'' represents $m-4$ arbitrarily chosen real functions).
\item $\overline{Y_m^0}$ is the closure of $Y_m^0$ in $C_0(\overline{\gO'})\times H^{1/2}(\dgO)^m$.
\end{itemize}
\end{definition}
We prove now an analog of Theorem \ref{T:XY}.
 \begin{theorem}\label{T:XYaet}
Let the sets $X_m$, $Y_m$, and $Y_m^0$ be as above.  Then,
\begin{enumerate}
\item all these sets are non-empty;
\item $X_m\subset Y_m^0$.
\end{enumerate}

\end{theorem}

\begin{proof}
As before, if we prove non-emptiness of $X_m$ and the second statement of the theorem, this will imply the first statement.

Non-emptiness of $X_m$ is shown by noticing that since $\nabla u_{0,x_j}=\boldej$, $(0,\mathbf{f}_0)$ belongs to $X_m$.

We now prove the second statement.
Let $(\sigma,f_1,\ldots, f_m)\in X_m$.
We use the following chain of deformations

\bea\label{E:deformation1aet}
(\sigma,f_1,\ldots,f_{m-1},f_m) &\leadsto& (\sigma,if^I_{\sigma,\boldrho_1},\ldots,if^I_{\sigma,\boldrho_1}
,if^I_{\sigma,\boldrho_2}) \\
\label{E:deformation2aet} &\leadsto& (0,if^I_{0,\boldrho_1},\ldots,if^I_{0,\boldrho_1}
,if^I_{0,\boldrho_2})\\
\label{E:deformation3aet} &\leadsto& (0,\mathbf{f}_0)\;. \eea
The first deformation (\ref{E:deformation1aet})
will be given by
\be\label{E:fjtaet}
(\sigma, f_{1,t} ,\ldots,f_{m,t})\quad 0\leq t\leq 1\;, \ee
where $f_{j,t} = (1-t)f_j + itf^I_{\sigma,\boldrho_1}$ when $j=1,\ldots,m-1$
and $f_{m,t} = (1-t)f_m + itf^I_{\sigma,\boldrho_2}$,
the second deformation (\ref{E:deformation2aet}) will be defined by letting $\sigma_t=(1-t)\sigma$, $0\leq t\leq1$,
and the third (\ref{E:deformation3aet}) will be defined in a way similar to (\ref{E:deformation1aet}).

We now specify the complex vectors $\boldrho_1$ and $\boldrho_2$.
Let $\boldk_1=\boldeth$, $\boldk_2=\boldet$, and $\kperp_1=\kperp_2=\boldeo$; that is,

\bea
\boldrho_1&=&\frac{\rho_1}{\sqrt2}(\boldeth+i\boldeo) \nonumber\\
\boldrho_2&=&\frac{\rho_2}{\sqrt2}(\boldet+i\boldeo)\;. \eea
We also let $\theta_l(x)=\rho_l\boldeo\cdot x/\sqrt2$ for $l=1,2$.
Furthermore we take $\rho_1$ and $\rho_2$ sufficiently large (as needed in the rest of the proof) and
rationally independent, and we also assume without loss of generality that $\gO$ does not
contain the origin.

We first claim that $\boldAsigf$ is left semi-Fredholm along deformation (\ref{E:deformation1aet}).
To do this, we will show that for every $(x,\xi)\in\overline{\gO''}\times S^2$
at least one of the individual operators
$A_{\sigma,f_1}, \ldots, A_{\sigma,f_{m-1}}$
has nonzero principal symbol.
According to (\ref{E:thebigzero}), this is the case if the vector

\be
\label{E:vector1} (\boldeo-2\xi_1\xi)\cos\theta_1+(\boldeth-2\xi_3\xi)\sin\theta_1 \ee
is nonzero and $\rho_1$ is taken sufficiently large.
Let

\be \omega_{\theta_1}=\left(\begin{array}{c}
\cos\theta_1\\
0\\
\sin\theta_1\\
\end{array}\right) 
\ee
Then the expression in (\ref{E:vector1}), is equal to

\be \omega_{\theta_1}-2(\omega_{\theta_1}\cdot\xi)\xi\;. \ee
If this were the zero vector, that would mean that $\omega_{\theta_1}$ and $2(\omega_{\theta_1}\cdot\xi)\xi$
are parallel unit vectors.  That would force $\omega_{\theta_1}\cdot\xi$
to be equal to $1/2$, meaning (since $|\omega_{\theta_1}|=|\xi|=1$) that $\omega_{\theta_1}$ and $\xi$ are not parallel.
Hence the vector in (\ref{E:vector1}) is nonzero, proving the claim that $\boldAsigf$ is left semi-Fredholm along the deformation (\ref{E:deformation1aet}).
(This argument shows that $\boldAsigf$ is left semi-Fredholm along the deformation (\ref{E:deformation3aet}) also.)

Next we examine deformation (\ref{E:deformation2aet}).  In order to show that $\boldAsigf$ is left semi-Fredholm
along this deformation, we claim that for every $(x,\xi)\in\overline{\gO''}\times S^2$
at least one of the individual operators
$A_{\sigma_t,f_{\boldrho_1^I}}$, $A_{\sigma_t,f_{\boldrho_2^I}}$, or $A_{\sigma_t,(f_{\boldrho_1^I},f_{\boldrho_2^I})}$
has nonzero principal symbol.  In order to prove this, it suffices to show that $\nabla u_{\boldrho_1}^I(x)$ and $\nabla u_{\boldrho_2}^I(x)$ are linearly independent
in a neighborhood of $\overline{\gO''}$.
These two gradients satisfy

\bea \nabla u_{\boldrho_1}^I \parallel  \left(\boldeo\cos \theta_1+\boldeth\sin\theta_1\right) \\
\nabla u_{\boldrho_2}^I \parallel \left(\boldeo\cos \theta_2+\boldet\sin\theta_2 \right)\;.\eea
The only way $\nabla u_{\boldrho_1}^I$ can lie in the $\boldeo\boldet$-plane is if $\sin\theta_1=0$.  But then
$\sin\theta_2\neq0$, as $x$ cannot be 0 and $\rho_1$ and $\rho_2$ are rationally independent.  Thus $\nabla u_{\boldrho_2}^I$ has nonzero $\boldet$-component, meaning $\nabla u_{\boldrho_2}^I$
doesn't lie in the $\boldet\boldeth$-plane.
This establishes the claim.

From this we conclude that $(\sigma,\mathbf{f})\in Y_m^0$.\end{proof}

We can now prove the main theorem of this section.

\begin{theorem}
 \label{T:AET}
\ben
\item Let $n=2$ and let $\mathbf{f}=(x_1,x_2)$.  Then the operator

\be \boldAsigf=
\left(
\ba{c}
A_{\sigma,x_1}\\
A_{\sigma,x_2}\\
A_{\sigma,(x_1,x_2)}
\ea
\right) :L^2(\gO')\rightarrow L^2(\gO')^3.
\ee
is injective as an operator from $L^2(\gO')$ into
$L^2(\gO')^3$ for an open dense set of $\sigma\in \Real \CzOprime$.
\item Let $n=3$.  Then the operator


\be \boldAsigf=
\left(\begin{array}{c}
A_{\sigma,f_1}\\
\vdots\\
A_{\sigma,f_{m-1}}\\
A_{\sigma,(f_1,f_m)}
\end{array}\right)
\ee
is injective as an operator from $L^2(\gO')$ into
$L^2(\gO')^m$ for an open dense set
of $(\sigma,\mathbf{f})\in \overline{X_m}$.
\een
\end{theorem}

\begin{proof} Again we prove each statement separately.

(1)\quad
According to Lemma \ref{L:dependence}, the operators $A_{\sigma,x_j}$ depend analytically
on $\sigma\in \CzOprime$
as operators mapping $\Lad$ into $L^1(\gO')$.
An argument very similar to the one in the proof of Lemma \ref{L:dependence}
shows that the dependence of $A_{\sigma,(x_1,x_2)}$ as an operator mapping $\Lad$ into $L^1(\gO')$ also is analytic\footnote{Indeed, the dependence of each $u_{\sigma,x_j}$ and $v^{(j)}$
on the pair $(\sigma,x_j)$ is analytic, and all other operations are algebraic or differentiation.}.
Hence the map

\bea \Lad &\rightarrow& L^1(\gO)^3 \nonumber\\
\sigma &\mapsto& \boldAsigf \nonumber\eea
is analytic.
By Lemma \ref{L:ana_wrong_spaces},
$\boldAsigf$ is an analytic family of operators mapping $L^2(\gO')$ into $L^2(\gO')^3$.  As proved in \cite{ale}, the gradients $\nabla u_{\sigma,x_1}$ and $\nabla u_{\sigma,x_2}$
are nowhere parallel in $\gO$.  By \cite[Theorem 3.6]{KuchStein}, $\boldAsigf\in\Phi(L^2(\gO'),L^2(\gO'))$ when $\sigma\in\CinfzOprime$.
Because the set of Fredholm operators is open in the operator norm topology,
there is an open dense set $V\subset \CzOprime$, containing all $\sigma\in\Real\CinfzOprime$, where the operators are also Fredholm.
Then the first statement of Theorem \ref{T:ZKKP2} applied to $\Real V$
(in the version of the first statement of Corollary \ref{C:real}) implies that there exists a set $W$, open and dense in $\mathrm{Re\:}\CzOprime$, where the operators are injective.
This proves the first statement.

%
%

(2)\quad Next we consider $n=3$ and
proceed to proving the second statement of the theorem.
According to Lemma \ref{L:dependence}, the operators $A_{\sigma,f_j}$ depend analytically
on $\sigma\in \CzOprime$
as operators mapping $\Lad$ into $L^1(\gO')$.
Again, an argument very similar to the one in the proof of Lemma \ref{L:dependence}
shows that the dependence of $A_{\sigma,(f_1,f_m)}$ as an operator mapping $\Lad$ into $L^1(\gO')$ also is analytic\footnote{As before, the dependence of each $u_{\sigma,f_j}$ and $v^{(j)}$
on the pair $(\sigma,f_j)$ is analytic, and all other operations are algebraic or differentiation.}.
Hence the map

\bea \Lad\times H^{1/2}(\dgO)^m&\rightarrow& L^1(\gO)^m \nonumber\\
(\sigma,f_1,\ldots,f_m) &\mapsto& \boldAsigf \nonumber\eea
is analytic.
By Lemma \ref{L:ana_wrong_spaces},
$\boldAsigf$ is an analytic family of operators on $Y_m^0$ mapping $L^2(\gO')$ into $L^2(\gO')^m$.
There exists a subset $V\subset \CzOprime\times H^{1/2}(\dgO)$, open and dense in $\overline{Y_m^0}$, such that $\boldAsigf\in \Phi_l(L^2(\gO'),L^2(\gO')^m)$ for $(\sigma,\mathbf{f})\in V$.
Since $X_m\subset Y_m^0$, we may assume that $V$ contains an open neighborhood of $X_m$ in $\Real(\CzOprime\times H^{1/2}(\dgO)^m)$.
Since the particular operator $A_{0,\mathbf{f}_0}$ is injective, $\Real Y_m^0$ contains a point $(\sigma,\mathbf{f})$ at which $\boldAsigf$ is injective.
Then the second statement of Theorem \ref{T:ZKKP2}
(in the version of the first statement of Corollary \ref{C:real}) applied to $\Real V$ implies that there exists a set $W$, open and dense in $\Real V$,
where the operators are injective.
Since the restriction of an open dense set to a dense topological subspace is still open dense in that subspace, $W$ is open dense in $\overline{X_m}$. \end{proof} \hfill $\Box$

\subsection{Quantitative Photoacoustic Tomography}\label{S:uniquenessQPAT}

The standard model for diffusive regime photon propagation in biological tissues is

\be\label{qpat}
\begin{cases}
L_{\sigma,\gamma}u:=-\nabla \cdot (e^\sigma\nabla u)+e^\gamma u=0 \\
u|_{\dgO}=f

\end{cases}
\ee
(see, e.g. \cite{wang2}).
Here $\sigma$ and $\gamma$ are the log-diffusion and log-attenuation coefficients, respectively.

The photoacoustic tomography (PAT) procedure, done first, provides one with the values inside $\gO$ of the function

\be\label{mq}
F(x)=\Gamma (x) e^{\gamma(x)} u(x) .
\ee
The function $\Gamma(x)$ is the so-called Gr\"{u}neisen coefficient\footnote{The Gr\"{u}neisen coefficient is in principle also not known, so one might want to include it as an unknown in the reconstruction procedure, e.g. \cite{BalRenQPAT}. We are not doing this here.
In \cite{BalRenQPAT} it was shown that only two out of the three unknown functions $\Gamma$, $\gamma$, and $\sigma$ can be recovered.}
describing the transfer of
electromagnetic energy into acoustic energy.  Here we assume $\Gamma(x)$ to be identically equal to 1.

This function is the initial data for quantitative photoacoustic tomography (QPAT), which strives to reconstruct the coefficients $\sigma$ and $\gamma$ from the data (\ref{mq}).

We will denote by $F_j(x)$, $j=1,2,\ldots ,2J$ the internal data (\ref{mq}) that correspond to
solutions of (\ref{qpat}) with different boundary data functions $f_j$.

For such a measurement $F_j$, the mapping
$(\sigma,\gamma)\rightarrow F_j$ is Fr\'{e}chet differentiable at a pair of smooth background coefficients $(\sigma_0,\gamma_0)$
as a map from $\Lad\times \left(\Lad\cap H^1_0(\gO')\right)\rightarrow H^1(\gO)$ (see \cite{KuchStein}). The derivative can be computed formally as before:

\bea \sigma&=&\sigma_0+\gep\rho \nonumber\\
\gamma&=&\gamma_0+\gep\nu \nonumber\\
u^{(j)}&=&u_0^{(j)}+\gep v^{(j)}+o(\gep) \eea
where $\rho\in\Lad$, $\nu\in\Lad\cap H^1_0(\gO')$. Substitution into (\ref{qpat}) shows that $v^{(j)}\in H^1(\gO)$ solves the boundary value problem

\be
\begin{cases}
-\nabla\cdot(\esigz\nabla v^{(j)})+\egz v^{(j)}
= \nabla\cdot(\rho\esigz\nabla u_0^{(j)})-\nu \egz u_0^{(j)}\\
 v^{(j)}|_{\dgO}=0.

 \end{cases}
\ee
We thus find that the differential of the mapping $F_j$ is
\bea
A_j(\rho,\nu):&=&dF_j (\rho,\nu)\nonumber\\
&=&\nu u_0^{(j)} -L^{-1}_{\sigma_0,\gamma_0}\left(\nu u_0^{(j)}\right)+L^{-1}_{\sigma_0,\gamma_0}\left(\nabla\cdot(\rho\esigz\nabla u_0^{(j)})\right).
\label{E:dF}
\eea
Here $L^{-1}_{\sigma_0,\gamma_0}$ refers to the inverse of $L_{\sigma_0,\gamma_0}$ on $\gO$ with a homogeneous Dirichlet boundary condition on $\dgO$.

We observe that the operator $A_j\in L(\Lad^2,L^2(\gO'))$ is well defined for any $\sigma_0,\gamma_0\in L^\infty(\gO)$ and $f_j\in H^{1/2}(\dgO)$ \footnote{In fact,
if $\tilde{\sigma}, \tilde{\gamma}\in L^\infty(\gO)$ are given functions
and we define $\tilde{L}_{\mathrm{ad}}^\infty(\gO):=\left\{\sigma\in L^\infty(\gO)\:|\:\sigma=\tilde{\sigma}\:\mathrm{on}\;\gO\backslash\gO'\right\}$,
then the Fr\'echet derivative of $F_j$, computed with respect to $\sigma_0,\gamma_0\in \tilde{L}_{\mathrm{ad}}^\infty(\gO)$, is exactly given by (\ref{E:dF}).
We will not use this fact, however.}.
The analytic dependence of the operator $A_j$ on $\sigma_0$, $\gamma_0$ and $f_j$ is given by the following lemma:

\begin{lemma}\label{L:qpatdependence}
 The map

 \bea L^\infty(\gO)\times L^\infty(\gO) \times H^{1/2}(\dgO) &\rightarrow& L\left(\Lad,L^2(\gO')\right)\nonumber\\
 \label{E:qpatdependence} (\sigma_0,\gamma_0,f_j)&\mapsto& A_j \eea
 is analytic.
\end{lemma}
The proof is given in section \ref{S:lemmas}.

%

We aim to establish uniqueness of reconstruction for $(\rho,\nu)$ from the data $(A_1(\rho,\nu),\ldots,A_{2J}(\rho,\nu))$
for an open dense set of real-valued background coefficients $(\sigma_0,\gamma_0)\in C(\overline{\gO})^2$ and boundary data $f_1,\ldots,f_{2J}\in H^{1/2}(\dgO)$.
In order to do this, we establish uniqueness first for a particular pair of background coefficients.
This is done in the following lemma.

\begin{lemma}
\label{L:QPATpoint}
Let $\gl>0$, and let $\esigz=\gl^{-2}$, $\egz=1$.
\ben
\item
 Let $n=2$, let three sets of boundary values in (\ref{qpat}) be given as

\bea f_{1,1}&=&e^{\gl x_1}\;, \nonumber\\
f_{1,2}&=&e^{\gl x_2}\;, \nonumber\\
\label{E:uniquepointbdry} f_{2,2}&=&e^{-\gl x_2}\;, \eea
and let $\gl$ be sufficiently small.
Then the corresponding data (\ref{E:dF})
uniquely determine $\rho$ and $\nu$.
\item Let $n=3$, let four sets of boundary values in (\ref{qpat}) be given as

\bea f_{1,1}&=&e^{\gl x_1}\;, \nonumber\\
f_{1,2}&=&e^{\gl x_2}\;, \nonumber\\
f_{2,2}&=&e^{-\gl x_2}\;, \nonumber\\
f_{3,3}&=&e^{\gl x_3}\;, \eea
and let $\gl$ be sufficiently small.
Then the data (\ref{E:dF})
uniquely determine $\rho$ and $\nu$.
\een
\end{lemma}

\begin{proof}

For simplicity let us denote the operator $L_{\sigma_0,\gamma_0}$ with these values of $\sigma_0$ and $\gamma_0$ by $L_\gl$.
Equation (\ref{qpat}) then becomes

\be\label{E:qpatgl}
\begin{cases}
L_\gl u:=\left(-\frac{1}{\gl^2}\Delta +1\right) u=0 \\
u|_{\dgO}=f

\end{cases}.
\ee

From equation (\ref{E:dF}), the Fr\'echet derivatives of the functionals $A_j$ satisfy

\be \label{E:frdiffgl}
L_\gl A_j(\rho,\nu)=-\frac{1}{\gl^2}\Delta(\nu u_0^{(j)})+\frac{1}{\gl^2}\nabla\cdot(\rho\nabla u_0^{(j)}).\ee

Some solutions to (\ref{E:qpatgl}) are given by $u=e^{\pm\gl x_i}$, as long as $f$ is taken to be the boundary value of this function.

We first concentrate on the case $n=2$.
Assume for the moment that $\rho\in \Lad\cap H^1_0(\gO)$ instead of just $\Lad$.

Using these data we obtain from equation (\ref{E:frdiffgl}) the three equations

\bea \label{E:larry} L_\gl A_{1,1}&=& -\frac{1}{\gl^2}\Delta(\nu e^{\gl x_1}) + \frac{1}{\gl}e^{\gl x_1}(\partial_1\rho+\gl\rho) \\
\label{E:moe} L_\gl A_{1,2}&=& -\frac{1}{\gl^2}\Delta(\nu e^{\gl x_2}) + \frac{1}{\gl}e^{\gl x_2}(\partial_2\rho+\gl\rho) \\
\label{E:curly} L_\gl A_{2,2}&=& -\frac{1}{\gl^2}\Delta(\nu e^{-\gl x_2}) - \frac{1}{\gl}e^{-\gl x_2}(\partial_2\rho-\gl\rho) \eea

From equation (\ref{E:moe}) we have

\be \label{E:nuequals}
\nu = \gl e^{-\gl x_2}\Delta^{-1}(e^{\gl x_2}\partial_2\rho + \gl e^{\gl x_2}\rho - \gl L_\gl A_{1,2}). \ee
Here $\Delta^{-1}$ means the inverse of the Laplacian on $\gO$ with a homogeneous Dirichlet boundary condition.
This inverse is a bounded operator from $H^s(\gO)$ into $H^{s+2}(\gO)$ for $s\geq -1$ \cite{taylor}.
Inserting this into equations (\ref{E:larry}) and (\ref{E:curly}) gives

\bea e^{\gl x_1}(\partial_1\rho + \gl \rho) - \Delta\left(e^{\gl(x_1-x_2)}\Delta^{-1}(e^{\gl x_2}\partial_2\rho
 +\gl e^{\gl x_2}\rho - L_\gl A_{1,2})\right) &=& L_\gl A_{1,1} \nonumber\\
-e^{-\gl x_2}(\partial_2\rho - \gl \rho) - \Delta\left(e^{-2\gl x_2}\Delta^{-1}(e^{\gl x_2}\partial_2\rho
 +\gl e^{\gl x_2}\rho - L_\gl A_{1,2})\right) &=& L_\gl A_{2,2} \nonumber\eea

Differentiating the first of these with respect to $x_1$ and using the identity $\Delta (uv)=u\Delta v+v\Delta u+2\nabla u\cdot\nabla v$,
we obtain

\bea \gl \partial_1 L_\gl A_{1,1} &=&e^{\gl x_1}\partial^2_1\rho 
+2\gl e^{\gl x_1}\partial_1\rho+\gl^2 e^{\gl x_1}\rho \nonumber\\
&-&\partial_1\bigg[\Delta e^{\gl(x_1-x_2)}\Delta^{-1}(e^{\gl x_2}\partial_2\rho +\gl e^{\gl x_2}\rho - L_\gl A_{1,2}) \nonumber\\
&+& e^{\gl(x_1-x_2)}(e^{\gl x_2}\partial_2\rho +\gl e^{\gl x_2}\rho - L_\gl A_{1,2}) \nonumber\\
&+&2\nabla  e^{\gl(x_1-x_2)} \cdot \nabla\left(\Delta^{-1} (e^{\gl x_2}\partial_2\rho +\gl e^{\gl x_2}\rho - L_\gl A_{1,2})\right)\bigg]
\eea

We collect terms that do not depend on $\rho$ on the left hand side, and consolidate terms left over that are multiplied by $\gl$ after differentiation:

\bea
\gl \partial_1 \bigg(L_\gl A_{1,1} &-& \Delta e^{\gl(x_1-x_2)}\Delta^{-1}L_\gl A_{1,2} \nonumber\\
&-&e^{\gl(x_1-x_2)}L_\gl A_{1,2} - 2\nabla  e^{\gl(x_1-x_2)} \cdot \nabla\Delta^{-1}L_\gl A_{1,2} \bigg) \nonumber\\
&=& e^{\gl x_1}\partial^2_1\rho - e^{\gl x_1}\partial_1\partial_2\rho \nonumber\\
\label{E:par1}
&+& O(\gl)_{H^1_0(\gO)\rightarrow H^1_0(\gO)}(\rho) + \sum_{i=1,2} O(\gl)_{L^2(\gO)\rightarrow L^2(\gO)}(\partial_i\rho) \eea

We next take minus the derivative of equation (\ref{E:curly}) with respect to $x_2$, giving

\bea
-\gl\partial_2 \bigg(L_\gl A_{2,2} &-& \Delta e^{-2\gl x_2}\Delta^{-1} L_\gl A_{1,2} \nonumber\\
&-& e^{-2\gl x_2} L_\gl A_{1,2} - 2\nabla e^{-2\gl x_2}\cdot\nabla\Delta^{-1} L_\gl A_{1,2} \bigg) \nonumber\\
&=& 2e^{-\gl x_2}\partial_2^2\rho \nonumber\\
\label{E:par2}
&+& O(\gl)_{H^1_0(\gO)\rightarrow H^1_0(\gO)}(\rho) + O(\gl)_{L^2(\gO)\rightarrow L^2(\gO)}(\partial_2\rho) \eea
Adding equations (\ref{E:par1}) and (\ref{E:par2}),  we obtain

\bea \label{E:bigequation}
e^{\gl x_1}\partial^2_1\rho - e^{\gl x_2}\partial_1\partial_2\rho +2e^{-\gl x_2}\partial_2^2\rho  
 &+& \sum_{i=1,2} O(\gl)_{L^2(\gO)\rightarrow L^2(\gO)}(\partial_i\rho) \nonumber\\
&+&O(\gl)_{H^1_0(\gO)\rightarrow H^1_0(\gO)}(\rho)\nonumber\\
&=& G(A_{1,1},A_{1,2},A_{2,2}) \eea
($G(A_{1,1},A_{1,2},A_{2,2})\in H^{-1}(\gO)$ is the sum of all terms in (\ref{E:par1}) and (\ref{E:par2}) containing $A_{1,1},A_{1,2}$, or $ A_{2,2}$.)

The operator

\be A_\gl=e^{\gl x_1}\partial^2_1 - e^{\gl x_2}\partial_1\partial_2 +2e^{-\gl x_2}\partial_2^2 \nonumber\ee
is elliptic on $\gO$ for $0\leq\gl<<1$, as can be seen by easily checking for $\gl=0$.  Let $c_1(\gl)<0$ be the largest eigenvalue of $A_\gl$.
As a consequence of Rayleigh's formula, $c_1(\gl)$ depends continuously on $\gl$.

Let $P_\gl:H^1(\gO)\rightarrow L^2(\gO)$ be the operator such that the $\gl$-dependent terms in equation (\ref{E:bigequation})
equal $\gl P_\gl\rho$.  Note from (\ref{E:par1}) and (\ref{E:par2}) that $\norm{P_\gl}_{H^1(\gO)\rightarrow L^2(\gO)}$ is bounded independent of $\gl$ for $0\leq\gl\leq 1$.
Then for any $u\in H^1(\gO)$,

\be (A_\gl u,u)+(\gl P_\gl u,u)\leq c_1(\gl)\norm{u}^2_{H^1(\gO)}+\gl\norm{P_\gl}\norm{u}^2_{H^1(\gO)}\leq \frac{c_1(\gl)}{2}\norm{u}^2_{H^1(\gO)} \ee
for $\gl$ sufficiently small.
By the Lax-Milgrim Theorem,
equation (\ref{E:bigequation}) has a unique solution for $\rho\in H^1_0(\gO)$ for this range of $\gl$.
Using equation (\ref{E:nuequals})
we get a unique solution for $\nu$ too.

Suppose now that $\rho\in\Lad$ only, instead of $\Lad\cap H^1_0(\gO')$.  By \cite[Theorem 4.1]{KuchStein} and elliptic regularity, any pair of functions $(\rho,\nu)$ in the kernel of the map (\ref{E:dF})
for the boundary data (\ref{E:uniquepointbdry}) must lie in $C^\infty_0(\gO')$.  In particular they must lie in $H^1_0(\gO')$.  Since we have proved that
the data uniquely determine $\rho$ for any $\rho\in H^1_0(\gO')$, there must be a unique solution for $\rho\in\Lad$ as well.
This proves the first statement.

%
%
Now let $n=3$.  A procedure similar to the one for 2 dimensions, using $f_{3,3}$ in an exactly analagous manner to $f_{1,1}$, yields the equation for $\rho$:

\bea \label{E:bigequation3}
e^{\gl x_1}\partial^2_1\rho - e^{\gl x_2}\partial_1\partial_2\rho &+&2e^{-\gl x_2}\partial_2^2\rho
-  e^{\gl x_2}\partial_2\partial_3\rho + e^{\gl x_3}\partial^2_3\rho  \nonumber\\                  
 &+& \sum_{i=1,2,3} O(\gl)_{L^2(\gO)\rightarrow L^2(\gO)}(\partial_i\rho) \nonumber\\
&+&O(\gl)_{H^1_0(\gO)\rightarrow H^1_0(\gO)}(\rho)\nonumber\\
&=& G(A_{1,1},A_{1,2},A_{2,2},A_{3,3}) \eea

($G(A_{1,1},A_{1,2},A_{2,2},A_{3,3})$ is again explicitly computable in a similar way to the $n=2$ case.)
Inspection of the first line shows this is an elliptic operator for $\gl$ sufficiently small, so (\ref{E:bigequation3}) has a unique
solution as before. \end{proof}

\begin{remark}\label{R:vectorfield}
Using the notation $u_{i,j}$ for the solution of (\ref{E:qpatgl}) with boundary data $f_{i,j}$
as in Lemma \ref{L:QPATpoint}, consider the vector fields formed from the pairs $u_{1,1},u_{1,2}$ and $u_{1,1},u_{2,2}$
as follows:

\bea V_1&=&u_{1,1}\nabla u_{1,2}-u_{1,2}\nabla u_{1,1} \nonumber\\
V_2 &=& u_{1,1}\nabla u_{2,2}-u_{2,2}\nabla u_{1,1}\;. \nonumber\eea
These vector fields are parallel to
$\boldeo-\boldet$ and $-\boldeo-\boldet$, respectively.  We note that $V_1$ and $V_2$ thus span $\R^2$.
Similarly, when $n=3$, the vector fields formed from the pairs $u_{1,1},u_{1,2}$ and $u_{1,1},u_{2,2}$, and $u_{1,1},u_{3,3}$, which are parallel to
$\boldeo-\boldet$, $-\mathbf{e_1}-\boldet$,
and $\boldeo-\boldeth$, span $\R^3$.
The same obviously holds true if the $f_{i,j}$ are multiplied by any constants.  \hfill $\Box$

The significance of the spanning condition on these vector fields was discussed in \cite{BalUhl} and later in \cite{KuchStein}.
\end{remark}

Let us fix $\gl>0$ small enough that the conclusions of Lemma \ref{L:QPATpoint} hold.
For convenience we change our notation slightly at this point.  We let $\sigma_0$ be such that $\esigz=\gl^{-2}$,
and we now denote the smooth background coefficients just by $\sigma$ and $\gamma$.

As in Section \ref{S:umot}, let $\chi\in C^\infty_0(\gO)$ be a cutoff function that is identically equal to 1 on $\overline{\gO''}$.
Let $\boldAsiggamf:\Lad\times(H^1_0(\gO')\cap\Lad)\rightarrow H^1_0(\gO)^{2J}$ be defined by

\be \boldAsiggamf=\chi\left(\begin{array}{c}A_1\\ \vdots\\ A_{2J}\end{array}\right)\chi\;. \ee
Because of the presence of the cutoff function $\chi$, $\boldAsiggamf$ can be viewed as an operator on $\R^n$.
It was shown in \cite{KuchStein} that $\boldAsiggamf$ is a pseudo-differential operator
with Douglis-Nirenberg parameters $s=(1,\ldots,1)$, $t=(0,1)$
and principal symbol

\bea
\boldAsiggamf(x,\xi)= \chi^2(x)\left(
  \begin{array}{cc}
    \frac{i\xi\cdot\nabla u^{(1,1)}(x)}{|\xi|^2} & u^{(1,1)}(x)  \\
 \frac{i\xi\cdot\nabla u^{(1,2)}(x)}{|\xi|^2} & u^{(1,2)}(x)  \\
 \vdots & \vdots\\
    \frac{i\xi\cdot\nabla u^{(J,1)}(x)}{|\xi|^2} & u^{(J,1)}(x)  \\
 \frac{i\xi\cdot\nabla u^{(J,2)}(x)}{|\xi|^2} & u^{(J,2)}(x)
  \end{array}\right),
\eea
Furthermore it was also shown that if at each $x\in\overline{\gO''}$ at least one of the 2 by 2 blocks

\be\label{E:determinants}
\left(
  \begin{array}{cc}
    \frac{i\xi\cdot\nabla u^{(j,1)}}{|\xi|^2} & u^{(j,1)}  \\
 \frac{i\xi\cdot\nabla u^{(j,2)}}{|\xi|^2} & u^{(j,2)}
  \end{array}\right)
\ee
is invertible, then $\boldAsiggamf$ is a left semi-Fredholm operator
from $L^2(\gO')\oplus H^1_0(\gO')$ into $H^1(\gO)^{2J}$.
This, in turn, is the case if the vector fields

\be\label{E:Vjqpat} V_j(x) := u^{(j,2)}(x)\nabla u^{(j,1)}(x) - u^{(j,1)}(x)\nabla u^{(j,2)}(x) \ee
span $\R^n$ at each point $x\in\overline{\gO''}$.

We define the following sets for $M\geq n$ and $L\geq3$:
\begin{definition}\indent
\begin{itemize}
\item $X_{ML}$ is the set of all real-valued triples
\bea
(\sigma,\gamma,\mathbf{f})&\in& \Real \left(C^\infty(\overline{\gO})^2\times H^{1/2}(\dgO)^{2ML}\right)\nonumber\\
\mathbf{f}&=&f^{(m,l,q)}\geq0\quad m=1,\ldots, M,\quad l=1,\ldots,L,\quad q=1,2\nonumber
\eea such that
the vector fields

\bea
\nonumber V_{1,1},&\ldots&,V_{M,L}\\
\nonumber V_{m,l}(x)&= &u^{(m,l,2)}(x)\nabla u^{(m,l,1)}(x) - u^{(m,l,1)}(x)\nabla u^{(m,l,2)}(x)
\eea
of the corresponding solutions of (\ref{qpat}) span the whole space $\R^n$ at every point
$x\in\overline{\gO''}$,
and such that, for each $m$ and $x$, the ratios $(u^{(m,1,1)}(x)/u^{(m,1,2)}(x))$ and $(u^{(m,l,1)}(x)/u^{(m,l,2)}(x))$
are not equal for at least one value of $l\geq3$.
\item $\overline{X_{ML}}$ is the closure of $X_{ML}$ in $\Real \left(C(\overline{\gO})^2\times  H^{1/2}(\dgO)^{2ML}\right)$.
\item $Y_{ML}$ is the set of (possibly complex-valued) triples
$$
(\sigma,\gamma,\mathbf{f})\in C^\infty(\overline{\gO})^2\times H^{1/2}(\dgO)^{2ML}
$$
 such that
 $$
 \boldAsiggamf\in\Phi_l(L^2(\gO')\oplus H^1_0(\gO'),H^1(\gO)^{2ML}).
 $$
 \item $Y_{ML}^0$ is the connected component of $Y_{ML}$ containing the point $(\sigma_0,0,\mathbf{f}_0)$, where
 $\mathbf{f}_0$ is an extension of the boundary data of Lemma \ref{L:QPATpoint}
 to a set of $2ML$ boundary data functions in such a way that $(\sigma_0,0,\mathbf{f}_0)$
 is contained in $X_{ML}$
 (that such a set of boundary data $\mathbf{f}_0$ exists will be part Theorem \ref{T:XYqpat}).
 \item $\overline{Y_{ML}^0}$ is the closure of $Y_{ML}^0$ in $C(\overline{\gO})^2\times H^{1/2}(\dgO)^m$.
\end{itemize}
\end{definition}

We note that the condition $f^{(m,l,q)}\geq0$ in the definition of $X_{ML}$ implies that the corresponding solutions
$u^{(m,l,q)}$ are all bounded below by a positive constant on $\overline{\gO''}$.

The following theorem is analogous to Theorems \ref{T:XY} and \ref{T:XYaet}.
\begin{theorem}\label{T:XYqpat}
\item Let $X_{ML}$, $Y_{ML}$ and $Y_{ML}^0$ be as above.  Then,
\begin{enumerate}
\item There exists a set of $2ML$ boundary data functions $\mathbf{f}_0$, extending the boundary data of Lemma \ref{L:QPATpoint},
such that $(\sigma_0,0,\mathbf{f}_0)\in X_{ML}$; 
\item $X_{ML}\subset Y_{ML}^0$.
\end{enumerate}
\end{theorem}

In particular, the first statement of Theorem \ref{T:XYqpat} implies that $X_{ML}$ is nonempty, 
and that $Y_{ML}^0$ can be well defined by
selecting an appropriate extension $\mathbf{f}_0$ of the boundary data of Lemma \ref{L:QPATpoint} and letting $Y_{ML}^0$ be the connected component of $Y_{ML}$ containing $(\sigma_0,0,\mathbf{f}_0)$.

\begin{proof}
We start by proving the first statement.
Let $n=2$, let $(\sigma,\gamma)=(\sigma_0,0)$, and let

\bea
f^{(1,1,1)}&=&e^{\gl x_1} \nonumber\\
f^{(1,1,2)}&=&e^{\gl x_2} \nonumber\\
f^{(1,3,1)}&=&e^{\gl x_1} \nonumber\\
f^{(1,3,2)}&=&e^{-\gl x_2-c}\;. \label{E:boundarynonempty}\eea
By Remark \ref{R:vectorfield},
the vector fields $V_{1,1}$ and $V_{1,3}$ span $\R^2$ at every $x\in\gO$,
so any extension of (\ref{E:boundarynonempty}) to a set of $2ML$ functions lies in $Y_{ML}$.  If $c$ is taken sufficiently large depending on $\gO$,
the ratios $u^{(1,1,1)}(x)/u^{(1,1,2)}(x)$ and $u^{(1,3,1)}(x)/u^{(1,3,2)}(x)$ are easily observed to be unequal for every $x\in\overline{\gO''}$.  By choosing a particular extension of these four boundary value functions
to a set of $2ML$ functions $\mathbf{f}_0$ in a way that keeps the necessary ratios unequal (e.g. by duplicating these functions indexed in a proper way),
we see that $(\sigma_0,0,\mathbf{f}_0)\in X_{ML}$.  Hence $X_{ML}$ is nonempty.

If $n=3$, the pair $(\sigma_0,0)$ along with the six functions

\bea
f^{(1,1,1)}&=&e^{\gl x_1} \nonumber\\
f^{(1,1,2)}&=&e^{\gl x_2} \nonumber\\
f^{(1,2,1)}&=&e^{\gl x_1} \nonumber\\
f^{(1,2,2)}&=&e^{\gl x_3} \nonumber\\
f^{(1,3,1)}&=&e^{\gl x_1} \nonumber\\
f^{(1,3,2)}&=&e^{-\gl x_2-c}\;, \nonumber\eea
extended appropriately to a set of $2ML$ functions $\mathbf{f}_0$ as in the $n=2$ case, is easily seen to belong to $Y_{ML}$ and $X_{ML}$.
This proves the first statement.

We now prove second statement.
The change of the unknown function $u\mapsto \sqrt{e^\sigma}u$ transforms the differential equation in (\ref{qpat}) into

\be L_qu:=(-\Delta+q(x))u=0 \ee
where $q=e^\gamma\Delta\sqrt{e^\sigma}/\sqrt{e^\sigma}$.
This equation has CGO solutions.

Let

\bea
\boldrho_{m,l,q}&=&\frac{\rho_{m,l,q}}{\sqrt2}(\boldk_m+i\kperp_m) \\
&=&
\begin{cases}
\frac{\rho_{m,l,q}}{\sqrt2}(\boldk_m+i\boldem)  \quad 1\leq m\leq n\\
\frac{\rho_{m,l,q}}{\sqrt2}(\boldk_m+i\mathbf{e}_n) \quad m>n
\end{cases}
\eea
where $\boldk_m$ can be chosen to be any vector perpendicular to $\kperp_m$.
For each $m$ we set $\boldrho_{m,1,1}=\boldrho_{m,l,1}$ and $\boldrho_{m,1,2}=\boldrho_{m,l,2}$ for $l\geq3$.
We take the $\rho_{m,l,q}$ to be similar in size (differing by at most 1, say),
rationally independent, and also such that the differences $\rho_{m,1,1}-\rho_{m,1,2}$
are rationally independent from $\rho_{m,2,1}-\rho_{m,2,2}$.  Let us also define

\be \rho:=\min_{m,l,q}\rho_{m,l,q}\;.\ee

As before, we define a chain of three deformations:

\bea \label{E:deformationqpat1}
(\sigma,\gamma,f^{(1,1,1)},\ldots,f^{(M,L,2)}) &\leadsto& (\sigma,\gamma,if^I_{\sigma,\gamma,\boldrho_{1,1,1}},\ldots,if^I_{\sigma,\gamma,\boldrho_{M,L,2}})\\
\label{E:deformationqpat2} &\leadsto& (\sigma_0,0,if^I_{\boldrho_{1,1,1}},\ldots,if^I_{\boldrho_{M,L,2}})\\
\label{E:deformationqpat3} &\leadsto& (\sigma_0,0,\mathbf{f}_0)\;. \eea
The first deformation (\ref{E:deformationqpat1})
is defined by
\be\label{E:fmlqt}
(\sigma,\gamma,\mathbf{f}_t)=(\sigma, \gamma,f_{1,1,1;t} ,\ldots,f_{M,L,2;t})\quad 0\leq t\leq 1\;, \ee
where $f_{m,l,q;t} = (1-t)f_{m,l,q} + itf^I_{\boldrho_{m,l,q}}$,
the second deformation (\ref{E:deformationqpat2}) is defined by letting

\bea
\sigma_t&=&(1-t)\sigma+t\sigma_0 \nonumber\\
\gamma_t&=&(1-t)\gamma\quad 0\leq t\leq1\;,\nonumber\eea
and the third (\ref{E:deformationqpat3}) is defined in a way similar to (\ref{E:deformationqpat1}).

The operator $\boldAsiggamf$ is left semi-Fredholm along the first and third deformations, which follows from the following lemma.

\begin{lemma}
Let $(\sigma,\gamma,\mathbf{f})\in X_{ML}$, and let $f_{m,l,q;t}$ be as in (\ref{E:fmlqt}).  Then
$\left(\sigma,\gamma,\mathbf{f}_t\right)\in Y_{ML}$ for $\rho$ being sufficiently large.

\end{lemma}

%

{\it Proof of the lemma.}
%



Let $u_t^{(m,l,q)}=(1-t)u^{(m,l,q)}+itu_{\boldrho_{m,l,q}}^I$ and $V_t^{(m,l)}$ be the vector field
formed by $u_t^{(m,l,1)}$ and $u_t^{(m,l,2)}$.  Then the imaginary part of $V_t^{(m,l)}$
is

\bea t(1-t)\Big(u_{\boldrho_{m,l,2}}^I\nabla u^{(m,l,1)} - u^{(m,l,1)}\nabla u_{\boldrho_{m,l,2}}^I\nonumber\\
\label{E:ImpartVt}
-u_{\boldrho_{m,l,1}}^I\nabla u^{(m,l,2)} + u^{(m,l,2)}\nabla u_{\boldrho_{m,l,1}}^I\Big)  \;.\eea
By the construction of a left regularizer in the proof of Theorem 4.1 in \cite{KuchStein}, the lemma
will be proved if we can show that these vector fields span $\R^n$ at each $x\in\overline{\gO''}$ for $m=1,\ldots,M$, $l=1,\ldots,L$.
Using (\ref{E:CGOgradbigger}) and the fact that the functions $u^{(m,l,q)}$ are bounded below by a positive constant on $\overline{\gO''}$, the vector field in (\ref{E:ImpartVt}) equals

\be t(1-t) \Big(- u^{(m,l,1)}\nabla u_{\boldrho_{m,l,2}}^I + u^{(m,l,2)}\nabla u_{\boldrho_{m,l,1}}^I\Big)\left(1+O\left(\frac{1}{\rho}\right)\right)\;. \ee
%
Hence, for $\rho$ sufficiently large, we have that
the imaginary parts of the vector fields $V_t^{(m,l)}$ span $\R^n$ at each point $x\in\overline{\gO''}$ if the vector fields

\be\label{E:eezgoodno}  - u^{(m,l,1)}\nabla u_{\boldrho_{m,l,2}}^I + u^{(m,l,2)}\nabla u_{\boldrho_{m,l,1}}^I \ee
span $\R^n$ at each point.
We recall that since

\be \nabla u^I_{\boldrho_{m,l,q}} = e^{\frac{\rho_{m,l,q}}{\sqrt2}\boldk_m\cdot x}
(\sin\theta_{m,l,q}\boldk_m+\cos\theta_{m,l,q}\kperp_m)\;,\ee
the vectors $\nabla u_{\boldrho_{m,l,1}}^I$ and $\nabla u_{\boldrho_{m,l,2}}^I$ span the $\boldk_m\boldem$-plane unless $\theta_{m,l,1}$
and $\theta_{m,l,2}$ differ by a factor of $\pi$, and for a given $x$ this can happen for at most one $m$ and $l$.
The requirement that
the ratios $(u^{(m,1,1)},u^{(m,1,2)})$ and $(u^{(m,l,1)},u^{(m,l,2)})$ be unequal for some $l\geq3$ ensures that the vector fields (\ref{E:eezgoodno})
are not all parallel, and so they span the $\boldk_m\boldem$-plane as well.  This proves the lemma.
        $\Box$\\

To see that $\boldAsiggamf$ is left semi-Fredholm along the second deformation (\ref{E:deformationqpat2})
we consider the vector fields $V_{ML}$ formed by the CGO solutions $u^I_{\boldrho_{m,l,q}}$.
To top order in $\rho$,

\bea V_{m,l} &=& e^{\sqrt2(\rho_{m,l,1}+\rho_{m,l,2})\boldk_m\cdot x}
\big(\sin\theta_{m,l,1}(\sin\theta_{m,l,2}\boldk_m+\cos\theta_{m,l,2}\kperp_m) \nonumber\\
&-&\sin\theta_{m,l,2}(\sin\theta_{m,l,1}\boldk_{m,l,1}+\cos\theta_{m,l,1}\kperp_m)\big)  \nonumber\\
\label{E:cgovml} &=&e^{\sqrt2(\rho_{m,l,1}+\rho_{m,l,2})\boldk_m\cdot x}\sin(\theta_{m,l,1}-\theta_{m,l,2})\kperp_m\;.
\eea
Since for each $x$ and pair $m,l$, $\sin(\theta_{m,l,1}-\theta_{m,l,2})$ and $\sin(\theta_{m,l,3}-\theta_{m,l,4})$ cannot both vanish,
the vector fields (\ref{E:cgovml}) span $\R^n$ at each point $x\in\gO$.
This proves that $\boldAsiggamf$ is left semi-Fredholm along (\ref{E:deformationqpat2}),
completing the proof of the second statement of Theorem \ref{T:XYqpat}.
$\;\Box$
\end{proof}


We are now ready to state and prove the main theorem of this section.

\begin{theorem}\label{T:QPAT}
Let $n=2$ or $3$, and let $M\geq n$, $L\geq3$.
 Then the operator $\boldAsiggamf$ is injective for an open dense set
 of $(\sigma,\gamma,\mathbf{f})\in \overline{X_{ML}}$.
\end{theorem}

\begin{proof}
As an immediate consequence of Lemma \ref{L:qpatdependence}, the operators $\boldAsiggamf$ depend analytically
on $(\sigma,\gamma,\mathbf{f})\in C(\overline{\gO})^2\times H^{1/2}(\dgO)^m$
as a family of operators mapping $\Lad^2$ into $L^2(\gO')^{2ML}$.  By Lemma \ref{L:ana_wrong_spaces},
$\boldAsiggamf$ is an analytic family of operators on $Y_{ML}^0$ mapping $L^2(\gO')\oplus H^1_0(\gO')$ into $H^1(\gO')^{2ML}$.
There exists a subset $V\subset C(\overline{\gO})^2\times H^{1/2}(\dgO)$, open and dense in $\overline{Y_{ML}^0}$, such that
$\boldAsiggamf\in \Phi_l(L^2(\gO')\oplus H^1_0(\gO'),H^1(\gO')^{2ML})$ for $(\sigma,\gamma,\mathbf{f})\in V$.
Since $X_{ML}\subset Y_{ML}^0$, we may assume that $V$ contains an open neighborhood of $X_{ML}$ in $\Real(C(\overline{\gO})\times H^{1/2}(\dgO)^m)$.
By Lemma \ref{L:QPATpoint} $A_{\sigma_0,0,\mathbf{f}_0}$ is an injective operator, so $\Real Y_{ML}^0$ contains a point at which $\boldAsiggamf$ is injective.
Then the second statement of Theorem \ref{T:ZKKP2}
(in the version of the first statement of Corollary \ref{C:real}) applied to $\Real V$ implies that there exists a set $W$, open and dense in $\Real V$,
where the operators are injective.
Since the restriction of an open dense set to a dense topological subspace is still open dense in that subspace, $W$ is open dense in $\overline{X_{ML}}$.\end{proof}
 \hfill $\Box$

\section{Proofs of some lemmas}\label{S:lemmas}

\subsection{Proof of Lemma \ref{L:positive}}

\begin{proof}
 Equation (\ref{E:maineq}) implies that $G_\mu$ is the Schwartz kernal of the operator mapping
 $f$ to $u$ in the boundary value problem (\ref{E:BVP_c}) (see, e.g., \cite{LionsNotes}).
 By elliptic regularity, $G_\mu(\eta,\cdot)$ is a smooth function on $\gO'$
 \cite{EvansPDE,taylor}.
 By the maximum principle, the operator $L_\mu^{-1}$ that maps a function $f$ to the solution
 $u$ of (\ref{E:BVP_c}) has the property that if $f\geq0$ then $u\geq0$ \cite{EvansPDE,taylor}.
 Thus $G_\mu$ is a nonnegative function.  Since $\gO'$ is compactly contained
 in $\gO$, the strong maximum principle
 implies that $G_\mu(\eta,\cdot)$ is bounded away from 0 on $\gO'$, proving the  lemma. \end{proof} \hfill $\Box$

\subsection{Proof of Lemma \ref{L:green}}

\begin{proof}

Let $L(X,Y)$ and $\mathrm{HS}(X,Y)$ denote the space of bounded operators and Hilbert-Schmidt operators, respectively, from $X$ to $Y$.

Consider the chain of maps

\bea \Lad &\rightarrow& L(H^2(\gO),L^2(\gO)) \rightarrow L(L^2(\gO),H^2(\gO))\nonumber\\
\label{E:fr} &\rightarrow& \mathrm{HS}(L^2(\gO),L^2(\gO)) \rightarrow L^2(\gO\times\gO)\\
\mu &\mapsto& L_\mu \mapsto L_\mu^{-1}\nonumber\\
 &\mapsto& L_\mu^{-1} \mapsto G .\eea
The last map in (\ref{E:fr}) is the mapping of an operator to its integral kernel.

The first two maps are Fr\'echet differentiable (as in \cite{KuchStein} for example).
Maps from $L^2(\gO)\rightarrow H^2(\gO)$ are Hilbert-Schmidt operators when considered as maps from $L^2(\gO)\rightarrow L^2(\gO)$
(see \cite{Mau}, Theorem 4).
By the Hilbert-Schmidt kernel theorem, the last map in (\ref{E:fr}) is a linear isomorphism.




Let $x\in\gO\backslash\gO''$.  By elliptic regularity applied to equation (\ref{E:maineqxitoo}), $G(x,\cdot)$
lies in $H^2(\gO')$, and $\norm{G(x,\cdot)}_{H^2(\gO)}
\leq C \norm{G(x,\cdot)}_{L^2(\gO)}$. (see \cite{Lions}, Theorem 2.3.2).  The constant $C$ is independent of $x$, as $x$ and $\xi$
are separated by a minimum positive distance, implying $G(x,\cdot)$ is a bounded function on $\gO'$.
Similarly, if $\xi\in\gO'$ is fixed, then $G(\cdot,\xi)$ solves the regular
boundary value problem in $\gO\backslash\gO''$

\bea (-\Delta +1)u&=&0 \nonumber\\
Bu|_{\partial\gO}&=&0\nonumber\\
Bu|_{\partial\gO''}&=&G(\cdot, \xi),\eea
as $\mu\equiv0$ outside $\gO'$.
Boundary elliptic regularity (as in \cite{Lions}, Theorem 2.5.1) gives us that $G(\cdot,\xi)$ lies in $H^2(\gO\backslash\gO'')$ and
$\norm{G(\cdot,\xi)}_{H^2(\gO\backslash\gO'')} \leq C \norm{G(\cdot,\xi)}_{L^2(\gO\backslash\gO'')}$ for a constant $C$ independent of $\xi$.
We have the corresponding estimate


\be
\label{E:embed}    \norm{G}_{H^2(\gO\backslash\gO''\times\gO')}
\leq C\norm{G}_{L^2(\gO\times\gO)}. \ee
The constant $C$ in (\ref{E:embed}) depends continuously on $\norm{\mu}_{L^\infty(\gO)}$, but this is a bounded quantity as we need only consider those $\mu$ which
deviate slightly from $\mu_0$.
Therefore, in the composition of maps

\bea \Lad&\rightarrow& L(L^2(\gO),H^2(\gO)) \rightarrow H^2(\gO\backslash\gO''\times\gO') \nonumber\\
\mu&\mapsto&L^{-1}_\mu\mapsto G\;,
\eea
the second map, which is linear, is continuous on the range of the first; hence the composition is Fr\'echet differentiable.


%
The lemma will thus follow once we establish the continuity of the map

\bea H^2(\gO\backslash\gO''\times\gO')&\rightarrow& H^2(\gO') \nonumber\\
\label{E:xionly} G&\mapsto& G(\eta,\cdot) .\eea
By the Sobolev embedding theorem (see e.g. \cite{taylor}, Proposition 4.4.3), $G(x,\cdot)\in C^{0,1/2}(\gO')$ for each fixed $x\in\gO\backslash\gO''$,
and $G(\cdot,\xi)\in C^{0,1/2}(\overline{\gO}\backslash\gO'')$ for each fixed $\xi\in\gO'$.
For any $x\in\gO\backslash\gO''$, we have

\bea |G(\eta,\xi)|&\leq& |G(x,\xi)|+\norm{G(\cdot,\xi)}_{C^{0,1/2}(\overline{\gO}\backslash\gO'')}|x-\eta|^{1/2} \nonumber\\
&\leq&|G(x,\xi)| +C(\gO)\norm{G(\cdot,\xi)}_{C^{0,1/2}(\overline{\gO}\backslash\gO'')}  .\eea
Averaging over $\overline{\gO}\backslash\gO''$ and using the Cauchy-Schwarz inequality we get

\bea |G(\eta,\xi)|&\leq& \frac{1}{\mathrm{Vol}(\overline{\gO}\backslash\gO'')}\int_{\overline{\gO}\backslash\gO''} |G(x,\xi)|\:dx \nonumber\\
&+& C(\gO) \norm{G(\cdot,\xi)}_{C^{0,1/2}(\overline{\gO}\backslash\gO'')}  \nonumber\\
&\leq& C(\gO,\gO'')\left(\norm{G(\cdot,\xi)}_{L^2(\overline{\gO}\backslash\gO'')}
+\norm{G(\cdot,\xi)}_{C^{0,1/2}(\overline{\gO}\backslash\gO'')}\right) \eea
Taking the $L^2(\gO')$-norm in $\xi$ and using elliptic regularity, we obtain the continuity of (\ref{E:xionly}). \end{proof} \hfill $\Box$


\subsection{Proof of Lemma \ref{L:dependence}}

\begin{proof}

First we note that the dependence of $u_{\sigma,f}\in H^1(\gO)$ on $\sigma$ is analytic (see, for example, \cite[Lemma 2.1]{KuchStein}),
and the dependence on $f$ is linear.  It remains to show that the dependence of the map

\bea \Lad&\rightarrow& H^1_0(\gO) \nonumber\\
\label{E:rho2v} \rho&\mapsto& v(\rho)
\eea
on $(\sigma,f)$ [as defined in equation (\ref{E:v})] is analytic, as all other operations in equation (\ref{E:Asigma}) are either algebraic or differentiation.
The operator $L_{\sigma}=-\nabla\cdot(e^{2\sigma/p}\nabla)\in L(H^1_0(\Omega),H^{-1}(\Omega))$ is invertible (see the simplest case of this statement in
\cite[Ch. 5, Proposition 1.1]{taylor} and general results in \cite{BerKrRo,Lions}); $L_{\sigma}$ is thus invertible in a neighborhood of $\sigma$ in $L^\infty(\gO)$.
This inverse $L_{\sigma}^{-1}:H^{-1}(\gO)\rightarrow H^1(\gO)$ depends analytically on $\sigma$, since $L_\sigma$ does, and the operation of taking the inverse of an
operator is known to be analytic on the domain of invertible operators (e.g., \cite{ZKKP}).  It is then evident from (\ref{E:v})
that the map (\ref{E:rho2v}) is analytic.\end{proof}  \hfill $\Box$


\subsection{Proof of Lemma \ref{L:qpatdependence}}

\begin{proof}
 As in the proof of Lemma \ref{L:dependence}, we first note that the dependence of $u_0^{(j)}\in H^1(\gO)$ on $(\sigma_0,\gamma_0)\in L^\infty(\gO)^2$ is analytic 
and the dependence on $f$ is linear.  The operator $L_{\sigma_0,\gamma_0}(\cdot)=-\nabla\cdot(e^{\sigma_0}\nabla)(\cdot)+e^{\gamma_0}(\cdot)\in L(H^1_0(\Omega),H^{-1}(\Omega))$ is invertible \cite{BerKrRo}.
Since the operation of taking the inverse of an operator is analytic, one observes that the expression in (\ref{E:dF})
depends analytically on $(\sigma_0,\gamma_0,f)$.  Hence the map (\ref{E:qpatdependence}) is analytic.
\end{proof}  \hfill $\Box$


\section{Remarks}\label{S:remarks}

\begin{enumerate}

\item
The reader notices that having the background coefficients lie in $C^\infty_0(\gO')$, as we do in Theorems \ref{T:testcase} and \ref{T:AET},
forces their values near the boundary to be constant. In turn, this allows us to work somewhat away from the boundary, which makes things simpler. One can generalize to the case of known (variable) values near the boundary,
e.g. by changing the definition of the space $\Lad$ to be the space of $L^\infty$-functions that equal some prescribed function near the boundary.
This is essentially what we do in Section \ref{S:uniquenessQPAT}, in considering only perturbations $\rho$ and $\nu$ that
are supported away from the boundary.
However, the theory of overdetermined elliptic boundary value problems (originated by \cite{Spencer}) has been well developed
(see, e.g. the books \cite{GudKrein,Samb} and paper \cite{Solon}).
This should allow one to relax this condition.
And indeed, this was partially done in \cite{BalInternal,IOut2,MonSt}.


\item Our goal was to prove genericity of linearized uniqueness, where ``genericity'' is understood in the strongest possible sense, namely ``except for an analytic subset.'' As we have already mentioned, doing so requires an alternative approach, such as working in the classes of pseudo-differential operators with symbols of finite smoothness (such as, e.g., in \cite{TaylorTools}). This is done in the next paper \cite{Stein3}, which will also contain some local (non-linear) uniqueness results.
\end{enumerate}

\section{Acknowledgments}\label{S:acknow}

The work of the first author was partly supported by the US NSF Grants DMS 0908208
and DMS 1211463, as well as by the DHS Grant 2008-DN-077-ARI018-04. The work of both authors was partially supported by KAUST through IAMCS. Thanks also go to Y.~Pinchover, P.~Stefanov, G.~Uhlmann, and T.~Widlak for helpful comments and references.


\bibliography{KuchStein2_D22}

\def\cprime{$'$}
\begin{thebibliography}{10}

\bibitem{Alb1}
G.~Alberti.
\newblock Enforcing local non-zero constraints in {PDE}s and applications to
  hybrid imaging problems.
\newblock 2014.
\newblock http://arxiv.org/abs/1406.3248.

\bibitem{Alb_thesis}
G.~Alberti.
\newblock {\em On local constraints and regularity of {PDE} in
  electromagnetics. Applications to hybrid imaging inverse problems}.
\newblock PhD thesis, University of Oxford, 2014.

\bibitem{AleGlobal}
G.~Alessandrini.
\newblock Global stability for a coupled physics inverse problem.
\newblock 2014.
\newblock http://arxiv.org/abs/1404.1275.

\bibitem{ale}
G.~Alessandrini and V.~Nesi.
\newblock Univalent {$\sigma$}-harmonic mappings: connections with
  quasiconformal mappings.
\newblock {\em J. Anal. Math.}, 90:197--215, 2003.

\bibitem{AllmBang}
M.~Allmaras and W.~Bangerth.
\newblock Reconstructions in ultrasound modulated optical tomography.
\newblock {\em J. Inverse Ill-Posed Probl.}, 19(6):801--823, 2011.

\bibitem{Ammari_book}
H.~Ammari, E.~Bonnetier, Y.~Capdeboscq, M.~Tanter, and M.~Fink.
\newblock Electrical impedance tomography by elastic deformation.
\newblock {\em SIAM J. Appl. Math.}, 68(6):1557--1573, 2008.

\bibitem{Ammari_AET}
H.~Ammari, J.~Garnier, and W.~Jing.
\newblock Resolution and stability analysis in acousto-electric imaging.
\newblock {\em Inverse Problems}, 28(8):084005, 20, 2012.

\bibitem{Ammari_elast}
H.~Ammari, A.~Waters, and H.~Zhang.
\newblock Stability analysis for magnetic resonance elastography.
\newblock 2014.
\newblock http://arxiv.org/abs/1409.5138.

\bibitem{arnold}
V.~I. Arnold, S.~M. Gusein-Zade, and A.~N. Varchenko.
\newblock {\em Singularities of differentiable maps. {V}olume 1 and 2}.
\newblock Modern Birkh\"auser Classics. Birkh\"auser/Springer, New York, 2012.
\newblock Classification of critical points, caustics and wave fronts,
  Translated from the Russian by Ian Porteous based on a previous translation
  by Mark Reynolds, Reprint of the 1985 edition.

\bibitem{BalExplicit}
G.~Bal.
\newblock Explicit reconstructions in {QPAT}, {QTAT}, {TE}, and {MRE}.
\newblock http://arxiv.org/abs/1202.3117.

\bibitem{BalHybridRedundant}
G.~Bal.
\newblock Hybrid inverse problems and redundant systems of partial differential
  equations.
\newblock In {\em Inverse Problems and Applications}, volume 615 of {\em
  Contemp. Math.}, pages 15--48. Amer. Math. Soc., Providence, RI.

\bibitem{bal1}
G.~Bal.
\newblock Cauchy problem for ultrasound-modulated {EIT}.
\newblock {\em Anal. PDE}, 6(4):751--775, 2013.

\bibitem{BalInternal}
G.~Bal.
\newblock Hybrid inverse problems and internal functionals.
\newblock In {\em Inverse problems and applications: inside out. {II}},
  volume~60 of {\em Math. Sci. Res. Inst. Publ.}, pages 325--368. Cambridge
  Univ. Press, Cambridge, 2013.

\bibitem{bal3}
G.~Bal, E.~Bonnetier, F.~Monard, and F.~Triki.
\newblock Inverse diffusion from knowledge of power densities.
\newblock {\em Inverse Probl. Imaging}, 7(2):353--375, 2013.

\bibitem{BalGuoAnisCurrent}
G.~Bal, C.~Guo, and F.~Monard.
\newblock Inverse anisotropic conductivity from internal current densities.
\newblock {\em Inverse Problems}, 30(2):025001, 21, 2014.

\bibitem{BalGuoMon}
G.~Bal, C.~Guo, and F.~Monard.
\newblock Linearized internal functionals for anisotropic conductivities.
\newblock {\em Inverse Probl. Imaging}, 8(1):1--22, 2014.

\bibitem{BalMosk}
G.~Bal and S.~Moskow.
\newblock Local inversions in ultrasound-modulated optical tomography.
\newblock {\em Inverse Problems}, 30(2):025005, 17, 2014.

\bibitem{BalNaetar}
G.~Bal, W.~Naetar, O.~Scherzer, and J.~Schotland.
\newblock The {L}evenberg-{M}arquardt iteration for numerical inversion of the
  power density operator.
\newblock {\em J. Inverse Ill-Posed Probl.}, 21(2):265--280, 2013.

\bibitem{bal4}
G.~Bal and K.~Ren.
\newblock Multi-source quantitative photoacoustic tomography in a diffusive
  regime.
\newblock {\em Inverse Problems}, 27(7):075003, 20, 2011.

\bibitem{BalRenQPAT}
G.~Bal, K.~Ren, G.~Uhlmann, and T.~Zhou.
\newblock Quantitative thermo-acoustics and related problems.
\newblock {\em Inverse Problems}, 27(5):055007, 15, 2011.

\bibitem{BalSchAOT}
G.~Bal and J.~C. Schotland.
\newblock Inverse scattering and acousto-optic imaging.
\newblock {\em Phys. Rev. Letters}, 104:043902, 2010.

\bibitem{BalSchBio}
G.~Bal and J.~C. Schotland.
\newblock Ultrasound-modulated bioluminescence tomography.
\newblock {\em Phys. Rev. E}, 89:031201, Mar 2014.

\bibitem{BalUhl}
G.~Bal and G.~Uhlmann.
\newblock Inverse diffusion theory of photoacoustics.
\newblock {\em Inverse Problems}, 26(8):085010, 20, 2010.

\bibitem{BalUhlmSolutions}
G.~Bal and G.~Uhlmann.
\newblock Reconstruction of coefficients in scalar second-order elliptic
  equations from knowledge of their solutions.
\newblock {\em Comm. Pure Appl. Math.}, 66(10):1629--1652, 2013.

\bibitem{BalZhouMaxwell}
G.~Bal and T.~Zhou.
\newblock Hybrid inverse problems for a system of {M}axwell's equations.
\newblock {\em Inverse Problems}, 30(5):055013, 17, 2014.

\bibitem{BerKrRo}
J.~M. Berezanski{\u\i}, S.~G. Kre{\u\i}n, and J.~A. Ro{\u\i}tberg.
\newblock A theorem on homeomorphisms and local increase of smoothness up to
  the boundary for solutions of elliptic equations.
\newblock {\em Dokl. Akad. Nauk SSSR}, 148:745--748, 1963.

\bibitem{cap}
Y.~Capdeboscq, J.~Fehrenbach, F.~de~Gournay, and O.~Kavian.
\newblock Imaging by modification: numerical reconstruction of local
  conductivities from corresponding power density measurements.
\newblock {\em SIAM J. Imaging Sci.}, 2(4):1003--1030, 2009.

\bibitem{CoxQPAT}
B.~Cox, T.~Tarvainen, and S.~Arridge.
\newblock Multiple illumination quantitative photoacoustic tomography using
  transport and diffusion models.
\newblock In {\em Tomography and inverse transport theory}, volume 559 of {\em
  Contemp. Math.}, pages 1--12. Amer. Math. Soc., Providence, RI, 2011.

\bibitem{Samb}
P.~I. Dudnikov and S.~N. Samborski.
\newblock Linear overdetermined systems of partial differential equations.
  {I}nitial and initial-boundary value problems [ {MR}1135115 (92m:35188)].
\newblock In {\em Partial differential equations, {VIII}}, volume~65 of {\em
  Encyclopaedia Math. Sci.}, pages 1--86. Springer, Berlin, 1996.

\bibitem{EvansPDE}
L.~C. Evans.
\newblock {\em Partial differential equations}, volume~19 of {\em Graduate
  Studies in Mathematics}.
\newblock American Mathematical Society, Providence, RI, second edition, 2010.

\bibitem{ScherAET}
B.~Gebauer and O.~Scherzer.
\newblock Impedance-acoustic tomography.
\newblock {\em SIAM J. Appl. Math.}, 69(2):565--576, 2008.

\bibitem{GoKr}
I.~C. Gohberg and M.~G. Krein.
\newblock The basic results on the defect numbers, the root numbers and the
  indices of linear operators.
\newblock {\em Magyar Tud. Akad. Mat. Fiz. Oszt. K\"ozl.}, 23(3-4):387--460,
  1977.
\newblock Translated from the Russian by K{\'a}roly Buz{\'a}si.

\bibitem{GudKrein}
I.~S. Gudovi{\v{c}} and S.~G. Kre{\u\i}n.
\newblock Boundary value problems for overdetermined systems of partial
  differential equations.
\newblock {\em Differencial\cprime nye Uravnenija i Primenen.---Trudy Sem.
  Processy Differentsial\cprime nye Uravneniya i ikh Primenenie}, (Vyp.
  9):1--145, 1974.

\bibitem{HoffKnu}
K.~Hoffmann and K.~Knudsen.
\newblock Iterative reconstruction methods for hybrid inverse problems in
  impedance tomography.
\newblock {\em Sensing and Imaging}, 15(1), 2014.

\bibitem{HormImpl}
L.~H{\"o}rmander.
\newblock Implicit function theorems.
\newblock University Lecture, 1977.

\bibitem{HormImpl2}
L.~H{\"o}rmander.
\newblock On the {N}ash-{M}oser implicit function theorem.
\newblock {\em Ann. Acad. Sci. Fenn. Ser. A I Math.}, 10:255--259, 1985.

\bibitem{Kato}
T.~Kato.
\newblock {\em Perturbation theory for linear operators}.
\newblock Classics in Mathematics. Springer-Verlag, Berlin, 1995.
\newblock Reprint of the 1980 edition.

\bibitem{Krein}
S.~G. Kre{\u\i}n.
\newblock {\em Linear equations in {B}anach spaces}.
\newblock Birkh\"auser Boston, Mass., 1982.
\newblock Translated from the Russian by A. Iacob, With an introduction by I.
  Gohberg.

\bibitem{Kuch_hybrid}
P.~Kuchment.
\newblock Mathematics of hybrid imaging. {A} brief review.
\newblock In I.~Sabadini and D.~C. Struppa, editors, {\em The mathematical
  legacy of {L}eon {E}hrenpreis}, pages 183--208. Springer-Verlag, 2012.

\bibitem{KuRadon}
P.~Kuchment.
\newblock {\em The {R}adon Transform and Medical Imaging}, volume~85 of {\em
  CBMS-NSF Regional Conference Series in Applied Mathematics}.
\newblock Society for Industrial and Applied Mathematics (SIAM), Philadelphia,
  PA, 2014.

\bibitem{KuKuAET}
P.~Kuchment and L.~Kunyansky.
\newblock 2{D} and 3{D} reconstructions in acousto-electric tomography.
\newblock {\em Inverse Problems}, 27(5):055013, 21, 2011.

\bibitem{KuchStein}
P.~Kuchment and D.~Steinhauer.
\newblock Stabilizing inverse problems by internal data.
\newblock {\em Inverse Problems}, 28(8):084007, 20, 2012.

\bibitem{Lang}
S.~Lang.
\newblock {\em Introduction to differentiable manifolds}.
\newblock Universitext. Springer-Verlag, New York, second edition, 2002.

\bibitem{lau}
R.~S. Laugesen.
\newblock Injectivity can fail for higher-dimensional harmonic extensions.
\newblock {\em Complex Variables Theory Appl.}, 28(4):357--369, 1996.

\bibitem{LionsNotes}
J.-L. Lions.
\newblock Lectures on elliptic partial differential equations.
\newblock Tata Institute of Fundamental Research, 1957.

\bibitem{Lions}
J.-L. Lions and E.~Magenes.
\newblock {\em Non-homogeneous boundary value problems and applications. {V}ol.
  {I}}.
\newblock Springer-Verlag, New York, 1972.
\newblock Translated from the French by P. Kenneth, Die Grundlehren der
  mathematischen Wissenschaften, Band 181.

\bibitem{Mau}
K.~Maurin.
\newblock Abbildungen vom {H}ilbert-{S}chmidtschen {T}ypus und ihre
  {A}nwendungen.
\newblock {\em Math. Scand.}, 9:359--371, 1961.

\bibitem{Monard_thesis}
F.~Monard.
\newblock {\em Taming unstable inverse problems: Mathematical routes toward
  high-resolution medical imaging modalities}.
\newblock PhD thesis, Columbia University, 2012.

\bibitem{MonBalD2}
F.~Monard and G.~Bal.
\newblock Inverse anisotropic diffusion from power density measurements in two
  dimensions.
\newblock {\em Inverse Problems}, 28(8):084001, 20, 2012.

\bibitem{MonBalInvDiff}
F.~Monard and G.~Bal.
\newblock Inverse diffusion problems with redundant internal information.
\newblock {\em Inverse Probl. Imaging}, 6(2):289--313, 2012.

\bibitem{MonBalN3}
F.~Monard and G.~Bal.
\newblock Inverse anisotropic conductivity from power densities in dimension
  {$n\geq3$}.
\newblock {\em Comm. Partial Differential Equations}, 38(7):1183--1207, 2013.

\bibitem{Montalto}
C.~Montalto.
\newblock Conductivity recovery from one component of the current density.
\newblock 2014.
\newblock http://arxiv.org/abs/1408.0423.

\bibitem{MonSt}
C.~Montalto and P.~Stefanov.
\newblock Stability of coupled-physics inverse problems with one internal
  measurement.
\newblock {\em Inverse Problems}, 29(12):125004, 13, 2013.

\bibitem{MosSch}
S.~Moskow and J.~Schotland.
\newblock Hybrid inverse problem for porous media.
\newblock In {\em Inverse problems and applications}, volume 615, pages
  255--260. American Mathamtical Society, 2012.
\newblock Conference in honor of Gunther Uhlmann. Stefanov, Plamen, Vasy,
  Andr\'as and Zworski, Maciej, editors.

\bibitem{Nash}
J.~Nash.
\newblock The imbedding problem for {R}iemannian manifolds.
\newblock {\em Ann. of Math. (2)}, 63:20--63, 1956.

\bibitem{Nirenberg72}
L.~Nirenberg.
\newblock An abstract form of the nonlinear {C}auchy-{K}owalewski theorem.
\newblock {\em J. Differential Geometry}, 6:561--576, 1972.
\newblock Collection of articles dedicated to S. S. Chern and D. C. Spencer on
  their sixtieth birthdays.

\bibitem{Nirenberg81}
L.~Nirenberg.
\newblock Variational and topological methods in nonlinear problems.
\newblock {\em Bull. Amer. Math. Soc. (N.S.)}, 4(3):267--302, 1981.

\bibitem{Ren}
K.~Ren, H.~Gao, and H.~Zhao.
\newblock A hybrid reconstruction method for quantitative {PAT}.
\newblock {\em SIAM J. Imaging Sci.}, 6(1):32--55, 2013.

\bibitem{Solon}
V.~A. Solonnikov.
\newblock Overdetermined elliptic boundary value problems.
\newblock {\em Zap. Nau\v cn. Sem. Leningrad. Otdel. Mat. Inst. Steklov.
  (LOMI)}, 21:112--158, 1971.

\bibitem{Spencer}
D.~C. Spencer.
\newblock Overdetermined systems of linear partial differential equations.
\newblock {\em Bull. Amer. Math. Soc.}, 75:179--239, 1969.

\bibitem{StUhl}
P.~Stefanov and G.~Uhlmann.
\newblock Linearizing non-linear inverse problems and an application to inverse
  backscattering.
\newblock {\em J. Funct. Anal.}, 256(9):2842--2866, 2009.

\bibitem{Stein3}
D.~Steinhauer.
\newblock Stabilizing inverse problems by internal data. {III}. {L}inearized
  and local uniqueness.
\newblock In preparation.

\bibitem{Tam}
A.~Tamasan, A.~Timonov, and J.~Veras.
\newblock Stable reconstruction of regular 1-harmonic maps with a given trace
  at the boundary.
\newblock {\em Applicable Analysis}, 2014.
\newblock Available online and to appear in print.

\bibitem{TaylorTools}
M.~E. Taylor.
\newblock {\em Tools for {PDE}}, volume~81 of {\em Mathematical Surveys and
  Monographs}.
\newblock American Mathematical Society, Providence, RI, 2000.
\newblock Pseudodifferential operators, paradifferential operators, and layer
  potentials.

\bibitem{taylor}
M.~E. Taylor.
\newblock {\em Partial differential equations {I} and {II}}, volume 115 of {\em
  Applied Mathematical Sciences}.
\newblock Springer, New York, second edition, 2011.

\bibitem{IOut2}
G.~Uhlmann, editor.
\newblock {\em Inside out {II}: Inverse problems and applications}.
\newblock Mathematical Sciences Research Institute Publications. Cambridge
  University Press, Cambridge, 2012.

\bibitem{Veras_thesis}
J.~Veras.
\newblock {\em Electrical conductivity imaging via boundary value problems for
  the 1-Laplacian}.
\newblock PhD thesis, University of Central Florida, 2014.

\bibitem{wang2}
L.~Wang and H.-I. Hu.
\newblock {\em Biomedical optics: principles and imaging}.
\newblock Wiley-Interscience, 2007.

\bibitem{widlak}
T.~Widlak and O.~Scherzer.
\newblock Hybrid tomography for conductivity imaging.
\newblock {\em Inverse Problems}, 28(8):084008, 28, 2012.

\bibitem{WidScherz}
T.~Widlak and O.~Scherzer.
\newblock Stability in the linearized problem of quantitative elastography.
\newblock 2014.
\newblock http://arxiv.org/abs/1406.0291.

\bibitem{ZKKP}
M.~G. Za{\u\i}denberg, S.~G. Kre{\u\i}n, P.~A. Ku{\v{c}}ment, and A.~A. Pankov.
\newblock Banach bundles and linear operators.
\newblock {\em Uspehi Mat. Nauk}, 30(5(185)):101--157, 1975.

\end{thebibliography}
\bibliographystyle{abbrv}

\end{document}